\documentclass[11 pt]{article}

\title{The codimension-one cohomology of $\SL_n \Z$}
\author{Thomas Church and Andrew Putman\thanks{The first author was supported by NSF grants DMS-1103807 and DMS-1350138 and the Alfred P.~Sloan Foundation. The second author was supported by NSF grant DMS-1255350 and the Alfred P.\ Sloan Foundation.}}

\usepackage{amsmath,amssymb,amsthm,amscd,amsfonts}
\usepackage{epsfig,pinlabel}
\usepackage[hmargin=1.25in, vmargin=1in]{geometry}
\usepackage[font=small,format=plain,labelfont=bf,up,textfont=it,up]{caption}
\usepackage{type1cm}
\usepackage{calc}
\usepackage{bm}
\usepackage{hyperref}
\usepackage{enumitem}
\usepackage[all,cmtip]{xy}
\usepackage{chngcntr}

\theoremstyle{plain}
\newtheorem{theorem}{Theorem}[section]
\newtheorem{maintheorem}{Theorem}
\newtheorem{proposition}[theorem]{Proposition}
\newtheorem{lemma}[theorem]{Lemma}

\newtheorem{corollary}[theorem]{Corollary}

\newtheorem{claims}{Claim}
\newtheorem{step}{Step}

\counterwithin*{claims}{section}
\counterwithin*{step}{section}

\theoremstyle{definition}
\newtheorem{remark}[theorem]{Remark}

\newtheorem{definition}[theorem]{Definition}

\newtheorem{example}[theorem]{Example}

\newlist{compactitem}{itemize}{3}
\setlist[compactitem]{nosep}
\setlist[compactitem,1]{label=\textbullet}
\setlist[compactitem,2]{label=--}
\setlist[compactitem,3]{label=\ensuremath{\ast}}

\newlist{compactdesc}{description}{3}
\setlist[compactdesc]{nosep}

\newlist{compactenum}{enumerate}{3}
\setlist[compactenum]{nosep}
\setlist[compactenum,1]{label=\arabic*.}
\setlist[compactenum,2]{label=(\alph*)}
\setlist[compactenum,3]{label=\roman*.}

\newtheorem{innerthm}{Theorem}
\newenvironment{theorem-prime}[1]{\renewcommand\theinnerthm{\ref*{#1}$'$}\innerthm}{\endinnerthm}
\DeclareMathOperator{\coker}{coker}

\DeclareMathOperator{\GL}{GL}
\DeclareMathOperator{\SL}{SL}

\newcommand\Z{\ensuremath{\mathbb{Z}}}
\newcommand\Q{\ensuremath{\mathbb{Q}}}
\newcommand\N{\ensuremath{\mathbb{N}}}

\newcommand\F{\ensuremath{\mathbb{F}}}

\DeclareMathOperator{\HH}{H}
\newcommand\RH{\ensuremath{\widetilde{\HH}}}

\DeclareMathOperator*{\Max}{max}

\DeclareMathOperator{\Aut}{Aut}

\DeclareMathOperator{\Ind}{Ind}
\DeclareMathOperator{\Span}{span}

\newcommand\Set[2]{\ensuremath{\{\text{#1 $|$ #2}\}}}

\newcommand\Figure[4]{
\begin{figure}[t]
\centering
\centerline{\psfig{file=#2,scale=#4}}
\caption{#3}
\label{#1}
\end{figure}}

\newcommand{\para}[1]{\bigskip\noindent\textbf{#1.}}
\renewcommand{\phi}{\varphi}
\newcommand\PartialBases{\ensuremath{\mathcal{B}}}
\newcommand{\B}{\PartialBases}
\newcommand\PartialBasesA{\ensuremath{\mathcal{BA}}}
\newcommand{\BA}{\PartialBasesA}
\DeclareMathOperator{\Rank}{rank}
\newcommand{\rank}{r}
\DeclareMathOperator{\ActualRank}{rank}
\DeclareMathOperator{\vcd}{vcd}
\newcommand\Tits{\ensuremath{\mathcal{T}}}
\newcommand\TitsP{\ensuremath{\mathbb{T}}} %for referring to the poset, in case we want to change it later
\DeclareMathOperator{\St}{St}
\DeclareMathOperator{\ModSym}{\mathcal{I}}
\newcommand\I{\ModSym}
\DeclareMathOperator{\Link}{Link}
\DeclareMathOperator{\Star}{Star}
\newcommand\Poset{\ensuremath{\mathcal{P}}}
\DeclareMathOperator{\Height}{ht}
\newcommand\abs[1]{\left\lvert#1\right\rvert}

\renewcommand\l[1]{#1^\pm}
\newcommand\tensor{\otimes}
\newcommand\abold{\mathbf{a}}
\newcommand\sign{\varepsilon}
\newcommand{\x}[1]{\langle\hspace{-2pt}\langle #1 \rangle\hspace{-2pt}\rangle} 
\newcommand\Tri{\mathcal{C}}

\newcommand\SLvcd{\binom{n}{2}}

\newcommand\onto{\twoheadrightarrow}
\newcommand\coloneq{\mathrel{\mathop:}\mkern-1.2mu=}

\newcommand{\arXiv}[1]{\href{http://arxiv.org/abs/#1}{arXiv:#1}}
\newcommand{\doi}[1]{\href{http://dx.doi.org/#1}{doi:#1}}
\newcommand{\myemail}[1]{\href{mailto:#1}{\nolinkurl{#1}}}

\newcommand\iso{\cong}

\newcommand\PLink{\ensuremath{\widehat{\Link}}}
\newcommand\pihat{\widehat{\pi}}

\begin{document}

\maketitle

\begin{abstract}
We prove that $\HH^{\binom{n}{2}-1}(\SL_n \Z;\Q) = 0$, where $\binom{n}{2}$ is the cohomological dimension
of $\SL_n \Z$, and similarly for $\GL_n \Z$.  We also prove analogous vanishing theorems 
for cohomology with coefficients in a rational representation
of the algebraic group $\GL_n$.  These theorems are derived from a presentation of the Steinberg
module for $\SL_n \Z$ whose generators are integral apartment classes, generalizing Manin's presentation for the Steinberg module of $\SL_2 \Z$. This presentation was originally constructed by Bykovskii.  We give a new
topological proof of it.
\end{abstract}

\section{Introduction}
\label{section:introduction}

The cohomology of $\SL_n \Z$ plays a fundamental role in many areas of mathematics.  The
Borel Stability Theorem \cite{BorelStability} determines $\HH^k(\SL_n \Z;\Q)$ when
$k$ is sufficiently small (conjecturally, for $k<n-1$). %(say, $k \leq n/4$)
%TC: reference if we need one: Remark 9.7 of Burgos Gil's book on the Beilinson regulator 
However, little is known outside this stable range.  
Recall that if $\Gamma$ is a virtually torsion-free group, the virtual cohomological
dimension of $\Gamma$ is
\[\vcd(\Gamma)\coloneq \max \Set{$k$}{$\HH^k(\Gamma;V\tensor \Q) \neq 0$ for some $\Gamma$-module $V$}.\]
Borel--Serre \cite{BorelSerreCorners} proved that
\[\vcd(\SL_n \Z)=\vcd(\GL_n\Z)= \SLvcd.\]
The cohomology of $\SL_n \Z$ in degrees near $\SLvcd$ is thus the ``most unstable'' cohomology.
In 1976, Lee--Szczarba~\cite[Theorem~1.3]{LeeSzczarbaCongruence} proved that
$\HH^{\SLvcd}(\SL_n\Z;\Q)=0$.    This vanishing was recently extended to $\HH^{\SLvcd}(\SL_n\Z;V_\lambda)=0$ for rational
representations $V_{\lambda}$ of the algebraic group $\GL_n$ by Church--Farb--Putman \cite{CFPSteinberg}.

Our main theorem concerns the cohomology in codimension $1$.  For $\lambda = (\lambda_1,\ldots,\lambda_n) \in \Z^n$ with
$\lambda_1 \geq \cdots \geq \lambda_n$, 
let $V_{\lambda}$ be the rational representation of $\GL_n \Q$
with highest weight $\lambda$.  Define $\|\lambda\| = \sum_{i=1}^{n} (\lambda_i - \lambda_n)$. 

\enlargethispage{\baselineskip}
\begin{maintheorem}[{\bf Codimension-one Vanishing Theorem}]
\label{maintheorem:vanishing}
For any rational representation~$V_\lambda$ of $\GL_n \Q$, we have
\[\HH^{\SLvcd-1}(\SL_n \Z;V_\lambda)=\HH^{\SLvcd-1}(\GL_n \Z;V_\lambda) = 0\quad \quad \text{for all }n\geq 3+\|\lambda\|.\hskip-27pt\]
In particular, 
\[\HH^{\SLvcd-1}(\SL_n \Z;\Q) = \HH^{\SLvcd-1}(\GL_n \Z;\Q) =0\quad \quad \text{for all }n\geq 3.\]
\end{maintheorem}

\begin{remark}
Theorem~\ref{maintheorem:vanishing} proves the $k=1$ case of a conjecture of Church--Farb--Putman \cite{ChurchFarbPutmanConjecture}
which asserts that for all $k \geq 0$, we have $\HH^{\SLvcd-k}(\SL_n \Z;\Q) = 0$ for $n>k+1$. %TC: We do in fact conjecture this stronger statement, and our vanishing range here does match the conjecture, so I'd like to keep the stronger statement.
\end{remark}

\paragraph{Steinberg module.}
To compute $\HH^{\SLvcd-k}(\SL_n\Z;\Q)$ for small $k$, the crucial object to understand is the 
\emph{Steinberg module}, which we now discuss.
The Tits building $\Tits_n$ for $\GL_n\Q$ is the simplicial complex
whose $p$-simplices are flags of subspaces
\[0 \subsetneq V_0 \subsetneq \cdots \subsetneq V_p \subsetneq \Q^n.\]   
By the Solomon-Tits theorem, $\Tits_n$ is homotopy equivalent to an infinite wedge of $(n-2)$-dimensional spheres.  
The group $\GL_n\Q$ naturally acts on $\Tits_n$ by simplicial automorphisms, and the \emph{Steinberg module} for $\GL_n\Q$ is the $\GL_n\Q$-module 
\[\St_n\coloneq \RH_{n-2}(\Tits_n;\Z).\] 

\paragraph{Borel--Serre duality.}
While $\SL_n\Z$ does not satisfy Poincar\'{e} duality, Borel--Serre~\cite{BorelSerreCorners} proved that it does satisfy 
\emph{virtual Bieri--Eckmann duality} with rational dualizing module $\St_n$.  This means that for any 
$\Q\SL_n\Z$-module $V$ and any $k\geq 0$, we have
\begin{equation}
\label{eq:introBEduality}
\HH^{\SLvcd-k}\big(\SL_n \Z;\,V\big) \iso \HH_k\big(\SL_n\Z;\, \St_n\otimes V\big).
\end{equation}
A similar result holds for $\GL_n \Z$ (with the same dualizing module $\St_n$).

\paragraph{Presentation of $\St_n$.} 
To prove Theorem~\ref{maintheorem:vanishing}, we compute the right-hand side of \eqref{eq:introBEduality} using
a presentation of the $\GL_n \Z$-module $\St_n$.  The following theorem was originally proved
by Bykovskii \cite{Bykovskii} and generalizes the presentation of $\St_2$ given by
Manin in 1972 \cite[Theorem 1.9]{Manin}.  The second purpose of this paper is to offer a new proof of it.

\begin{maintheorem}[{\bf Presentation of \boldmath$\St_n$}]
\label{maintheorem:bykovskii}
For $n \geq 2$, the Steinberg module $\St_n$ is the abelian group with generators 
$[v_1,\ldots,v_n]$, one for each ordered basis $\{v_1,\ldots,v_n\}$
of $\Z^n$, subject to the following three families of relations.
\begin{compactenum}[label={\normalfont R\arabic*.},ref={R\arabic*}]
\item \label{rel:manin} \ \ $[v_1,v_2,v_3,\ldots,v_n]=[v_1,v_1+v_2,v_3,\ldots,v_n]+[v_1+v_2,v_2,v_3,\ldots,v_n]$.
\item \label{rel:sign} \ \ $[\pm v_1, \pm v_2,\ldots,\pm v_n]=[v_1,v_2,\ldots,v_n]$ for any choices of signs.
\item \label{rel:perm} \ \ $[v_{\sigma(1)},v_{\sigma(2)},\ldots,v_{\sigma(n)}]=(-1)^\sigma\cdot [v_1,v_2,\ldots,v_n]$ for any $\sigma\in S_n$.
\end{compactenum}
The $\GL_n\Z$-action on $\St_n$ is defined by $M\cdot[v_1,\ldots,v_n]=[M\cdot v_1,\ldots,M\cdot v_n]$.
\end{maintheorem}
The relation~\ref{rel:sign} can be replaced by the single relation $[-v_1,v_2,\ldots,v_n]=[v_1,v_2,\ldots,v_n]$. Relations~\ref{rel:manin} and \ref{rel:sign} then involve \emph{only the first two vectors}; these relations are the ``stabilization'' of Manin's relations for $\St_2$  from $\GL_2\Z$ to $\GL_n\Z$. In this light, what Theorem~\ref{maintheorem:bykovskii} says is that \emph{no additional relations} are needed to present $\St_n$, once the permutations $S_n$ are taken into account. 

\para{Apartment classes and Bykovskii's proof of Theorem~\ref{maintheorem:bykovskii}} 
The general theory of spherical buildings automatically provides a generating set for $\St_n$.  Namely, every rational basis $\{w_1,\ldots,w_n\}$ for $\Q^n$ determines a spherical apartment in $\Tits_n$ homeomorphic to $S^{n-2}$ whose fundamental class determines an apartment class $[w_1,\ldots,w_n]\in \RH_{n-2}(\Tits_n;\Z)=\St_n$, and the general theory implies that $\St_n$ is generated by these rational apartment classes. However, the generating set in Theorem~\ref{maintheorem:bykovskii} is much smaller: it consists only of the \emph{integral} apartment classes $[v_1,\ldots,v_n]$, i.e.\ those for which $\{v_1,\ldots,v_n\}$ is an integral basis for $\Z^n$. That $\St_n$ is generated by these integral apartment classes was proved by Ash--Rudolph \cite{AshRudolph} in 1979. To do this, they gave an algorithm for expressing an arbitrary rational apartment class as a sum of integral apartment classes.

Bykovskii proved Theorem~\ref{maintheorem:bykovskii} by carefully examining Ash--Rudolph's
algorithm, which requires making many arbitrary choices, and showing that the only ambiguity in its output comes from the relations in Theorem~\ref{maintheorem:bykovskii}. We remark that from this perspective, Theorem~\ref{maintheorem:bykovskii} appears as the integral analogue of Lee--Szczarba's presentation of $\St_n$ as a $\GL_n\Q$-module \cite{LeeSzczarbaCongruence}.

%\begin{remark}
%The reader might notice that Bykovskii's paper \cite{Bykovskii} is significantly shorter than this one.  His
%proof is not actually shorter than ours; rather, he sketches the details very briefly.  If it were written
%in the same level of detail as our paper, it would be a similar length.
%\end{remark}

\para{Our proof of Theorem~\ref{maintheorem:bykovskii}}
Our proof of Theorem~\ref{maintheorem:bykovskii} is quite different. It is inspired by our alternate
proof of Ash--Rudolph's theorem in Church--Farb--Putman \cite{CFPSteinberg} and by Manin's original proof of Theorem~\ref{maintheorem:bykovskii} for $\St_2$.  We use topology to show directly that the homology of $\Tits_n$ is generated by integral apartment classes; non-integral apartment classes never show up in our proof.  The key is the \emph{complex of partial augmented frames} for $\Z^n$ defined below, which provides an ``integral model''
for the Tits building $\Tits_n$. We begin first with the more familiar complex of partial frames.

\begin{definition}
Let $V$ be a finite-rank free abelian group.
\begin{compactitem}[leftmargin=25pt]
\item A \emph{line} in $V$ is a $2$-element set $\{v,-v\}$ of primitive vectors in $V$; we denote it by $\l{v}$.
\item A \emph{frame} for $V$ is a set $\{\l{v_1},\ldots,\l{v_n}\}$ of lines such that $\{v_1,\ldots,v_n\}$ is a basis for $V$. %; to see that this is well-defined, observe that replacing $v_i$ with $-v_i$ does not change whether or not $\{v_1,\ldots,v_n\}$ is a basis. 
\item A \emph{partial frame} for $V$ is a frame for a direct summand of $V$, or equivalently a set of lines in $V$ that can be completed to a frame for $V$.  %We will sometimes call a frame a \emph{complete frame} to emphasize that it is not merely a partial frame.
\end{compactitem}
The \emph{complex of partial frames} for $\Z^n$, denoted $\B_n$, is the simplicial complex
whose $p$-simplices are partial frames for $\Z^n$ of cardinality~$(p+1)$.
\end{definition}

The complex $\B_n$ is $(n-1)$-dimensional, and Maazen~\cite{MaazenThesis} proved that $\B_n$ is $(n-2)$-connected. This connectivity is what we used in \cite{CFPSteinberg} to prove Ash--Rudolph's theorem on generators for $\St_n$. However, to obtain a \emph{presentation} for $\St_n$ this is not enough; we need to attach higher-dimensional cells to $\B_n$ to improve its connectivity.

\paragraph{Improving connectivity: the complex of partial augmented frames.}
To motivate the cells we add, we recall how Manin found his presentation for $\St_2$.

The first key step is to show that in an appropriate sense, the first homology $\HH_1(\B_2;\Z)$ measures exactly the \emph{additional} relations beyond \ref{rel:sign} and \ref{rel:perm} needed to present $\St_2$.  This requires two observations:
\begin{compactenum}
\item The Tits building $\Tits_2$ can be identified with the $0$-skeleton of $\B_2$, 
giving an identification of the reduced chains $\widetilde{C}_0(\B_2;\Z)$ with 
$\St_2=\RH_0(\Tits_2;\Z)=\widetilde{C}_0(\Tits_2;\Z)$.
\item The complex $\B_2$ is the following graph:
\Figure{figure:farey}{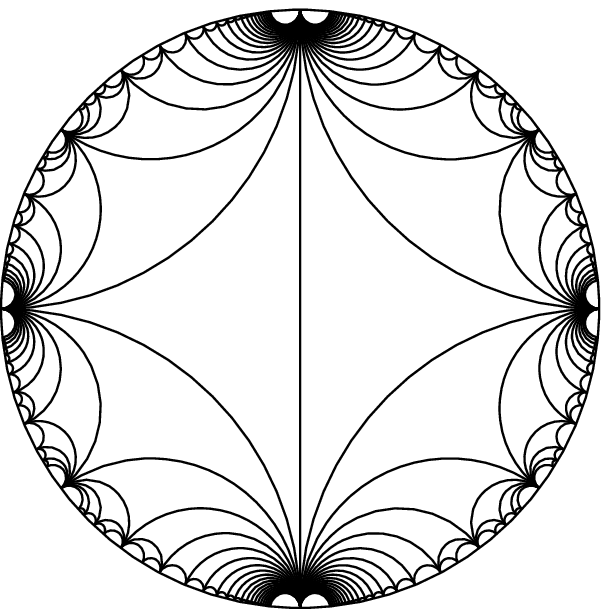}{The complex $\B_2$ is isomorphic to the Farey graph under the identification
that takes the line $\l{(a,b)}\in \Z^2$ to $\frac{a}{b} \in \Q \cup \{\infty\}$.}{100}
\begin{compactitem}
\item The vertices are lines $\l{(a,b)}$, where $(a,b) \in \Z^2$ is a primitive vector.
\item Vertices $\l{(a,b)}$ and $\l{(c,d)}$ are joined by an edge exactly when $\{(a,b),(c,d)\}$ is a basis
for $\Z^2$, or equivalently when $ad-bc = \pm 1$.
\end{compactitem}
If we identify the line $\l{(a,b)}\in \Z^2$ with $\frac{a}{b} \in \Q \cup \{\infty\}$, the complex $\B_2$ is exactly the classical Farey graph; see Figure \ref{figure:farey}. This graph is connected, but not simply-connected.
\end{compactenum}
Combining these two observations, we see that
\[\HH_1(\B_2;\Z)=\ker\big(\,C_1(\B_2;\Z)\onto \widetilde{C}_0(\B_2;\Z)\cong \St_2\,\big).\]
The group $C_1(\B_2;\Z)$ is precisely the abelian group given 
by the presentation with the same generators as in Theorem~\ref{maintheorem:bykovskii}, 
but where we impose only the relations \ref{rel:sign} and \ref{rel:perm}.  Indeed, this
is simply the fact that each ordered basis $\{v_1,v_2\}$ of $\Z^2$ determines an edge of $\B_2$, and this correspondence is unique up to negating (R2) or exchanging (R3) the basis vectors.
We conclude that 
$\HH_1(\B_2;\Z)$ measures whatever {additional} relations beyond \ref{rel:sign} and \ref{rel:perm} are needed to present $\St_2$, as claimed.

The second key step is that a visual examination of Figure \ref{figure:farey} suggests a natural generating set for $\HH_1(\B_2;\Z)$, namely the boundaries of the evident triangles in the Farey
graph.  Under our identification of $\B_2$ with the Farey graph, these triangles consist of triples of vertices
$\{\l{v_1}, \l{v_2}, \l{v_3}\}$ such that $\{v_1,v_2\}$ is a basis for $\Z^2$ and $\pm v_1 \pm v_2 \pm v_3 = 0$
for some choice of signs (in which case $\{v_2,v_3\}$ and $\{v_1,v_3\}$ are also bases
for $\Z^2$).  By reordering and negating we can assume that $v_3 = v_1+v_2$.  
The relation  in $\St_2$ corresponding to the boundary of the triangle $\{\l{v_1},\l{v_2},\l{(v_1+v_2)}\}$ is
\[[v_1,v_1+v_2] - [v_1,v_2] + [v_1+v_2,v_2] = 0,\] 
which is precisely the relation \ref{rel:manin}.
Manin's theorem that the relations \ref{rel:manin} together with \ref{rel:sign} and \ref{rel:perm} suffice
to present $\St_2$ thus follows from the fact that attaching the above triangles to the Farey graph
yields a simply-connected simplicial complex.  In fact, it yields a contractible complex.

This motivates the following definition.

\begin{definition}
Let $V$ be a finite-rank free abelian group.
\begin{compactitem}[leftmargin=25pt]
\item An \emph{augmented frame} for $V$ 
is a collection $\{\l{v_0},\l{v_1},\ldots,\l{v_n}\}$ of lines in $V$
such that $\{\l{v_1},\ldots,\l{v_n}\}$ is a frame for $V$ and $\pm v_0 \pm v_1 \pm v_2 = 0$ for some choice of signs.
\item A \emph{partial augmented frame} for $V$ is a set of lines in $V$ that is either a frame or an
augmented frame for a direct summand of $V$; equivalently, a set of lines is a partial augmented frame for $V$ if it can be completed to an augmented frame for $V$.
\end{compactitem}
The \emph{complex of partial augmented frames} for $\Z^n$, denoted $\BA_n$, is the simplicial complex whose 
$p$-simplices are the partial augmented frames for $\Z^n$ of cardinality~$(p+1)$.
\end{definition}

The final main theorem of this paper is as follows; the definition of a Cohen--Macaulay complex is recalled in \S\ref{section:posets} below.  We remark that this theorem plays a fundamental role in the second author's recent work with Day on
the second homology group of the Torelli subgroup of $\Aut(F_n)$; see \cite{DayPutmanH2}.

\begin{maintheorem}[{\bf $\PartialBasesA_n$ is Cohen--Macaulay}]
\label{maintheorem:basesacon}
For all $n\geq 2$, the complex $\PartialBasesA_n$ is Cohen--Macaulay of dimension $n$. In particular, $\BA_n$ is $(n-1)$-connected.
\end{maintheorem}

\begin{remark} Since $\BA_n$ is $n$-dimensional, the connectivity in Theorem~\ref{maintheorem:basesacon} cannot be improved unless $\BA_n$ is contractible.  Since $\BA_2$ is the complex obtained by filling in the triangles in the Farey graph,
the complex $\BA_2$ is contractible. 
However, it seems unlikely that $\BA_n$ would be
contractible for any $n \geq 3$.
\end{remark}

\para{Outline} The logical relation between our three main theorems is that Theorem~\ref{maintheorem:basesacon} $\implies$  Theorem~\ref{maintheorem:bykovskii} $\implies$ Theorem~\ref{maintheorem:vanishing}. However, the proof of Theorem~\ref{maintheorem:basesacon} occupies more than half of the paper, so we defer the proof of Theorem~\ref{maintheorem:basesacon} until \S\ref{section:proofbasesacon}.
We prove Theorem~\ref{maintheorem:bykovskii} in \S\ref{section:bykovskiiproof} and prove Theorem~\ref{maintheorem:vanishing} in \S\ref{section:vanishing}, both assuming Theorem~\ref{maintheorem:basesacon}.

\para{Acknowledgments} We are very grateful to Benson Farb, who was closely involved during the development of these results, but declined to be listed as a coauthor.

\section{Generators and relations for the Steinberg module}
\label{section:bykovskiiproof}

In this section, we derive Theorem~\ref{maintheorem:bykovskii} from Theorem~\ref{maintheorem:basesacon}, which
will be proved in \S\ref{section:proofbasesacon}.  We begin in \S\ref{section:posets} with some basic
results about posets.  We then prove some results about linear algebra over $\Z$ in \S \ref{section:linearz}.
Finally, we prove Theorem~\ref{maintheorem:bykovskii} in \S\ref{section:bykovskiiproofsec}.

\subsection{The topology of posets}
\label{section:posets}
Recall that a $d$-dimensional complex is \emph{$d$-spherical} if it is $(d-1)$-connected, in which case it is homotopy equivalent to a wedge of $d$-spheres.   
A simplicial complex $X$ is \emph{Cohen--Macaulay} (abbreviated CM) of dimension $d$ if the following conditions hold.
\begin{compactitem}
\item $X$ is $d$-spherical.
\item For every $(k-1)$-simplex $\sigma^{k-1}$ of $X$, the link $\Link_X(\sigma^{k-1})$
is $(d-k)$-spherical.
\end{compactitem}

\begin{remark}
This should be compared with the definition of a combinatorial $d$-manifold, which is a 
$d$-dimensional simplicial complex $M$ such that for every $(k-1)$-simplex $\sigma^{k-1}$, the link $\Link_M(\sigma^{k-1})$ is a combinatorial $(d-k)$-sphere. %removed ``(this definition is inductive; the base case is that a $0$-dimensional combinatorial manifold is a discrete set of points)'' b/c unnecessary, the condition just becomes vacuous.
\end{remark}

Let $A$ be a poset.  Recall that the \emph{geometric realization} of $A$ is the simplicial complex $\abs{A}$ whose
$k$-simplices are chains $a_0 \lneq a_1 \lneq \cdots \lneq a_k$
in $A$.  Whenever we say that $A$ has some topological property, we mean that $\abs{A}$ has that topological
property.  In particular, we define $\HH_{\ast}(A)$ to equal $\HH_{\ast}(\abs{A})$.  The following is a key
example.

\begin{example}
Let $X$ be a simplicial complex.  Define $\Poset(X)$ to be the poset of simplices of $X$ under inclusion.  Then
$\abs{\Poset(X)}$ is isomorphic to the barycentric subdivision of $X$.  In particular, there is a canonical
isomorphism $\HH_{\ast}(\Poset(X)) \cong \HH_{\ast}(X)$.
\end{example}

For $a \in A$, the \emph{height} of $a$, denoted $\Height(a)$, is the maximal $k$ such that there exists
a chain $a_0 \lneq a_1 \lneq \cdots \lneq a_k = a$ in $A$.
If $B$ is another poset and $F\colon A \rightarrow B$ is a poset map, for $b \in B$ we define
$F_{\leq b}$ to be the subposet $\Set{$a \in A$}{$F(a) \leq b$}$ of $A$.  With these definitions, we
have the following proposition, which slightly generalizes a result of Quillen.

\begin{proposition}
\label{prop:quillen} 
Fix $m\geq 0$ and let $F\colon A\to B$ be a map of posets.  Assume that $B$ is CM of dimension $d$ 
and that for all $b\in B$, the fiber $F_{\leq b}$ is $(\Height(b)+m)$-spherical (or more generally, that $\RH_q(F_{\leq b})=0$ for $q\neq \Height(b)+m$).
%TC: Please do not remove this parenthetical, the invocation of this proposition later is not valid without it.  
Then $F\colon A\to B$ is $(d+m)$-acyclic. In particular, $F_{\ast}\colon \RH_d(A)\to \RH_d(B)$ 
is an isomorphism if $m\geq 1$.
\end{proposition}\begin{proof}
This can be proved exactly like \cite[Theorem 9.1]{QuillenPoset}. The necessary conditions on $B$ are satisfied since it is CM. Quillen's hypothesis that $F_{\leq b}$ is $\Height(b)$-spherical is used only to conclude that $\RH_q(F_{\leq b})=0$ for $q\neq \Height(b)$, so we can replace this with the hypothesis that $\RH_q(F_{\leq b})=0$ for $q\neq \Height(b)+m$. We conclude that the spectral sequence  $E^2_{pq}\implies \HH_{p+q}(A)$ of \cite[(9.3)]{QuillenPoset} vanishes in the range $p+q<d+m$, except for $E^2_{d0}=\HH_d(B)$. Therefore $F\colon A\to B$ is $(d+m)$-acyclic, as claimed.
\end{proof}

\subsection{Linear algebra over \texorpdfstring{$\Z$}{Z}}
\label{section:linearz}

We now record some simple facts about linear algebra over $\Z$ that are classical and well-known to experts (but whose
proofs we include for completeness).

\begin{lemma}
\label{lemma:intersect}
If $V$ is a subspace of $\Q^n$, then $V \cap \Z^n$ is a direct summand of $\Z^n$.
\end{lemma}
\begin{proof}
Write $V$ as $V= \ker(\psi)$ for some linear map $\psi\colon\Q^n \rightarrow \Q^n$, and set $W = \psi(\Z^n) \subset \Q^n$. Since $V \cap \Z^n = \ker(\psi|_{\Z^n})$, we have a short exact sequence
\[0 \longrightarrow V \cap \Z^n \longrightarrow \Z^n \longrightarrow W \longrightarrow 0\]
of $\Z$-modules.  Since $W$ is a $\Z$-submodule of $\Q^n$, it must be torsion-free.  From this and the fact
that $W$ is finitely generated, we deduce that $W$ is a free $\Z$-module, and hence the above short exact sequence
splits.  The lemma follows.
\end{proof}

\begin{corollary}
\label{corollary:bijection}
Let $\mathcal{X}$ be the set of subspaces of $\Q^n$ and let
$\mathcal{X}'$ be the set of direct summands of $\Z^n$.  Then the map $\mathcal{X} \rightarrow \mathcal{X}'$ taking
$V \in \mathcal{X}$ to $V \cap \Z^n \in \mathcal{X}'$ is a bijection.
\end{corollary}
\begin{proof}
Lemma \ref{lemma:intersect} implies that the indicated map lands in $\mathcal{X}'$; the inverse is the map $\mathcal{X}'\to \mathcal{X}$ taking $A\in \mathcal{X}'$ to $A\otimes_\Z \Q\in \mathcal{X}$.
\end{proof}

\begin{lemma}
\label{lemma:transitive}
Let $A$ and $B$ be direct summands of $\Z^n$ such that $A \subset B$.  Then $A$ is a direct
summand of $B$.
\end{lemma}
\begin{proof}
Since $A$ is a direct summand of $\Z^n$, the quotient $\Z^n / A$ is torsion-free.  The
quotient $B/A$ is torsion-free, being contained in $\Z^n/A$. Since $B/A$ is finitely generated, it is in fact free. Hence the short exact sequence
\[0 \longrightarrow A \longrightarrow B \longrightarrow B/A \longrightarrow 0\]
splits, as desired.
\end{proof}

\subsection{The proof of Theorem~\texorpdfstring{\ref{maintheorem:bykovskii}}{B}}
\label{section:bykovskiiproofsec}

We now prove Theorem~\ref{maintheorem:bykovskii}.  During our proof
of Theorem~\ref{maintheorem:bykovskii}, we will use Theorem~\ref{maintheorem:basesacon} in two different places; the proof of Theorem~\ref{maintheorem:basesacon} is postponed until \S\ref{section:proofbasesacon}.
%Recall that Theorem~\ref{maintheorem:basesacon} asserts for $n\geq 2$ that $\BA_n$ is CM of dimension $n$.

\paragraph{The subcomplex $\BA'_n$.}
We begin by defining a subcomplex $\BA'_n$ of $\BA_n$.  Consider a simplex $\sigma = \{\l{v_1},\ldots,\l{v_k}\}$ of
$\BA_n$.  By definition, the submodule $\Span_\Z(v_1,\ldots,v_k)$ of $\Z^n$ is a direct summand (of rank $k$
if $\sigma$ is a partial frame and of rank $(k-1)$ otherwise).  Let $\BA'_n$
be the subcomplex of $\BA_n$ consisting of simplices $\sigma = \{\l{v_1},\ldots,\l{v_k}\}$ of $\BA_n$ such that
$\Span_\Z(v_1,\ldots,v_k)$ is a \emph{proper} direct summand of $\Z^n$.  The only simplices that are
omitted from $\BA'_n$ are the frames of $\Z^n$ (which are $(n-1)$-simplices) and the augmented frames of $\Z^n$ (which are $n$-simplices).

The proof now has three main steps.

\begin{step}
\label{step:identifypresentation}
The abelian group described by the presentation in Theorem~\ref{maintheorem:bykovskii} coincides with the relative homology $\HH_{n-1}(\BA_n,\BA'_n;\Z)$.
\end{step}

To prove this, we will compute the relative homology via the relative simplicial chain complex $C_{\ast}(\BA_n,\BA'_n)$. 
Since the $(n-2)$-skeleton of $\BA_n$ is contained in $\BA'_n$, we have $C_k(\BA_n,\BA'_n)=0$ for $k\leq n-2$.  It
follows that
\[\HH_{n-1}(\BA_n,\BA'_n;\Z) = \coker(C_n(\BA_n,\BA'_n) \xrightarrow{\partial} C_{n-1}(\BA_n,\BA'_n)).\]
Define $\I_0 = C_{n-1}(\BA_n,\BA'_n)$ and $\I_1 = C_n(\BA_n,\BA'_n)$.  Our goal is to describe $\I_0$ and
$\I_1$ and the differential $\partial$.
% Removed some \coloneq

The simplices that contribute to $\I_0$ are the $(n-1)$-simplices of $\BA_n$ that do not lie in $\BA'_n$, i.e.\ those
corresponding to frames $\{\l{v_1},\ldots,\l{v_n}\}$ for $\Z^n$.  To specify such a frame, it is
enough to give the vectors $v_1,\ldots,v_n$. The only ambiguity is that multiplying the vectors $v_i$ by $\pm 1$
does not change the frame, nor does permuting the vectors; however, permuting the vectors does change the orientation of the corresponding simplex.  We deduce that $\I_0$ is the abelian
group with generators the set of formal symbols $\x{v_1\ldots,v_n}$ for bases $\{v_1,\ldots,v_n\}$ of $\Z^n$ subject
to the following relations.
\begin{compactenum}[label={\normalfont S\arabic*.},ref={S\arabic*},start=2]
\item \label{rel:Ssign} $\x{\pm v_1, \pm v_2,\ldots,\pm v_n}=\x{v_1,v_2,\ldots,v_n}$ for any choices of signs.
\item \label{rel:Sperm} $\x{v_{\sigma(1)},v_{\sigma(2)},\ldots,v_{\sigma(n)}}=(-1)^\sigma\cdot \x{v_1,v_2,\ldots,v_n}$ for any $\sigma\in S_n$.
\end{compactenum}

The simplices that contribute to $\I_1$ are the $n$-simplices of $\BA_n$ that do not lie in $\BA'_n$ (a vacuous
condition since $\BA'_n$ does not contain any $n$-simplices).  These correspond to augmented frames
$\{\l{v_0},\ldots,\l{v_n}\}$ for $\Z^n$.  By definition, $\{v_1,\ldots,v_n\}$ is a basis for $\Z^n$ and
$\pm v_0 \pm v_1 \pm v_2 = 0$ for some choice of signs.  Multiplying $v_0$ and $v_1$ and $v_2$ by
appropriate choices of signs, we can arrange for $v_0 = v_1 + v_2$.  For $3 \leq i \leq n$, the set
$\{v_0,\ldots,\widehat{v}_i,\ldots,v_n\}$ spans a proper direct summand of $\Z^n$, so this term of the boundary vanishes in $\I_0$.  This implies that under the boundary map $\partial$, the generator of $\I_1$ corresponding to the augmented frame
$\{\l{(v_1+v_2)},\l{v_1}\ldots,\l{v_n}\}$ has image in $\I_0$ equal to \[\x{v_1,\ldots,v_n} - \x{v_1+v_2,v_2,\ldots,v_n} + \x{v_1+v_2,v_1,v_3,\ldots,v_n}.\]
Applying the relation \ref{rel:Ssign} and rearranging, we see that $\HH_{n-1}(\BA_n,\BA'_n;\Z) = \coker(\partial)$ is the quotient of $\I_0$ by
the set of relations
\begin{compactenum}[label={\normalfont S\arabic*.},ref={S\arabic*}]
\item\label{rel:Smanin} $\x{v_1,v_2,v_3,\ldots,v_n}=\x{v_1,v_1+v_2,v_3,\ldots,v_n}+\x{v_1+v_2,v_2,v_3,\ldots,v_n}$.
\end{compactenum}
The relations \ref{rel:Smanin}, \ref{rel:Ssign}, and \ref{rel:Sperm} correspond exactly to the relations \ref{rel:manin}, \ref{rel:sign}, and \ref{rel:perm} in Theorem~\ref{maintheorem:bykovskii}, yielding the identity claimed in Step~\ref{step:identifypresentation}.

\begin{step}
\label{step:simplifyidentification}
We have $\HH_{n-1}(\BA_n,\BA'_n;\Z)\iso\RH_{n-2}(\BA'_n;\Z)$.
\end{step}

This is our first invocation of Theorem~\ref{maintheorem:basesacon}, %, which we recall is being assumed and which will be proved in \S\ref{section:proofbasesacon}.
which states that $\BA_n$ is CM of dimension $n$. 
In particular,
\[\HH_{n-1}(\BA_n;\Z) = \HH_{n-2}(\BA_n;\Z) = 0.\]
From the long exact sequence for relative homology, we obtain the desired isomorphism.

\begin{step}
\label{step:finalidentification}
We have $\RH_{n-2}(\BA'_n;\Z) \cong \St_n$.
\end{step}

Let $\TitsP_n$ denote the poset of proper nonzero direct summands of $\Z^n$ under inclusion.
By Corollary \ref{corollary:bijection}, the poset $\TitsP_n$
is isomorphic to the poset of proper nonzero $\Q$-subspaces of $\Z^n$, so its geometric realization $\abs{\TitsP_n}$ can be identified with the Tits building $\Tits_n$ from the introduction, whose homology is the Steinberg module.  In other words,
\[\RH_{n-2}(\TitsP_n;\Z)\iso \RH_{n-2}(\Tits_n;\Z)= \St_n.\]
Recall that $\Poset(\BA'_n)$ is the poset of simplices of $\BA'_n$.  There is a poset map
\[F\colon \Poset(\BA'_n) \rightarrow \TitsP_n\]
defined by
\[F(\{\l{v_1},\ldots,\l{v_k}\}) = \Span_\Z(\l{v_1},\ldots,\l{v_k}).\]
We remark that $F$ can only be defined on $\BA'_n$ and not on $\BA_n$ since $\TitsP_n$ consists
of \emph{proper} direct summands.  To prove the isomorphism claimed in Step~\ref{step:finalidentification}, we will use
Proposition~\ref{prop:quillen} (with $d = n-2$ and $m=1$) to prove that $F$ induces an isomorphism 
$F_{\ast}\colon \RH_{n-2}(\BA'_n) \xrightarrow{\iso} \RH_{n-2}(\TitsP_n;\Z)$.

We need to verify that $F$ satisfies the conditions of Proposition~\ref{prop:quillen}.
The Solomon-Tits theorem 
states that $\TitsP_n$ is CM of dimension $n-2$ (see e.g. \cite[Remark IV.4.3]{BrownBuildings} or \ \cite[Example~8.2]{QuillenPoset}). It remains to verify the second condition of Proposition~\ref{prop:quillen} for each direct summand $V \in \TitsP_n$.

If $\ActualRank(V)=1$, the fiber $F_{\leq V}$ is easy to describe: a rank-1 direct summand $V$ contains only one line, so $F_{\leq V}$ is a single point. In particular, the hypothesis $\RH_q(\F_{\leq V})=0$ holds for \emph{all} $q$ in this case.

Now consider a direct summand $V$ with $\ActualRank(V)=\ell\geq 2$.  A partial augmented frame $\{\l{v_1},\ldots,\l{v_k}\}$ lies in $F_{\leq V}$ if and only if $\Span_\Z(\l{v_1},\ldots,\l{v_k})$ is contained in $V$; by Lemma \ref{lemma:transitive}, this holds if and only if $\Span_\Z(\l{v_1},\ldots,\l{v_k})$ is a direct summand of $V$. In other words, the poset $F_{\leq V}$ consists of those collections of lines that form a partial augmented frame for $V$.  Choosing an isomorphism $V\iso \Z^\ell$, we therefore obtain an identification 
\[F_{\leq V} \iso \BA_\ell.\]
We now invoke Theorem~\ref{maintheorem:basesacon} for the second time: it states that $\BA_\ell$ is $\ell$-spherical, so $F_{\leq V}$ is $\ActualRank(V)$-spherical. Since $\Height(V)=\ActualRank(V)-1$, this verifies the desired hypothesis for $F_{\leq V}$.

\section{The vanishing theorem}
\label{section:vanishing}

In this section, we use Theorem~\ref{maintheorem:bykovskii} to prove Theorem~\ref{maintheorem:vanishing}. The actual proof
is contained in \S\ref{section:vanishingproof}.  This is preceded by 
\S\ref{section:vanishinglemmas}, which contains some preliminary lemmas.

\subsection{Ingredients of the vanishing theorem}
\label{section:vanishinglemmas}

This section contains two ingredients needed for the proof of Theorem
\ref{maintheorem:vanishing}.  The first is as follows.

\begin{lemma}
\label{lemma:flatresolution}
Let $G$ be a group and let $M$ and $N$ be $G$-modules.  Assume that $N$ is a
vector space over a field of characteristic $0$.  Also, let
\[\cdots \longrightarrow F_2 \longrightarrow F_1 \longrightarrow F_0 \longrightarrow M \longrightarrow 0\]
be a resolution of $M$ by flat $G$-modules.  Then the homology of the chain complex
\[\cdots \longrightarrow F_2 \otimes_G N \longrightarrow F_1 \otimes_G N \longrightarrow F_0 \otimes_G N \longrightarrow 0\]
equals $\HH_{\ast}(G;M \otimes N)$, where $G$ acts diagonally on $M \otimes N$.
\end{lemma}
\begin{proof}
This combines the statements of \cite[Proposition III.2.1]{BrownCohomology} and
\cite[Proposition III.2.2]{BrownCohomology}.
\end{proof}

To make Lemma~\ref{lemma:flatresolution} useful, we need a simple way of recognizing
flat $G$-modules.  Our second lemma is such a criterion.

\begin{lemma}
\label{lemma:recognizeflat}
Let $G$ be a group, let $X$ be a simplicial complex on which $G$ acts simplicially,
and let $Y$ be a subcomplex of $X$ which is preserved by the $G$-action.  For some $n \geq 0$,
assume that the setwise stabilizer subgroup $G_{\sigma}$ is finite 
for every $n$-simplex $\sigma$ of $X$ that is not contained in $Y$.  Then
the $G$-module $C_n(X,Y;\Q)$ of relative simplicial $n$-chains is flat.
\end{lemma}
\begin{proof}
For an oriented $n$-simplex $\sigma$ of $X$ that is not contained in $Y$, let
$[\sigma]$ be the associated basis element of $C_n(X,Y;\Q)$.  Define
$M_{\sigma} \subset C_n(X,Y;\Q)$ to be the span of $\Set{$[g(\sigma)]$}{$g \in G$}$,
so $M_{\sigma}$ is a $G$-submodule of $C_n(X,Y;\Q)$.  As in the statement of the lemma,
$G_{\sigma}$ will denote the setwise stabilizer subgroup of $\sigma$. This subgroup
may reverse the orientation of $\sigma$.  Let $\Q_{\sigma}$ be the
$G_{\sigma}$-module whose underlying vector space is $\Q$ but where an
element of $G_{\sigma}$ acts by $\pm 1$ depending on whether or not it reverses
the orientation of $\sigma$.  We then have that
\[M_{\sigma} \cong \Ind_{G_{\sigma}}^{G} \Q_{\sigma}.\]
Since $\Q_{\sigma}$ is an irreducible representation of the finite group
$G_{\sigma}$, it is a direct summand of $\Q[G_{\sigma}]$.  It follows
that $M_{\sigma}$ is a direct summand of 
\[\Ind_{G_{\sigma}}^{G} \Q[G_{\sigma}] \cong \Q[G].\]
Since $\Q[G]$ is a localization of the free $G$-module $\Z[G]$, it is a flat $G$-module.
We deduce that $M_{\sigma}$ is a flat $G$-module.
Choosing representatives for the $G$-orbits of $n$-simplices of $X$
not lying in $Y$ determines an isomorphism 
\[C_n(X,Y;\Q) \iso \bigoplus_{\sigma \in (X^{(n)}-Y^{(n)})/G} M_{\sigma},\]
so $C_n(X,Y;\Q)$ is a flat $G$-module, as desired.
\end{proof}

\subsection{The proof of Theorem~\texorpdfstring{\ref{maintheorem:vanishing}}{A}}
\label{section:vanishingproof}

We now prove Theorem~\ref{maintheorem:vanishing}.  We begin by recalling its
statement.  Fixing some $\lambda \in \Z^n$ and some
$n \geq 3+\|\lambda\|$, this theorem asserts that
\[\HH^{\SLvcd-1}(\SL_n \Z;V_\lambda)=\HH^{\SLvcd-1}(\GL_n \Z;V_\lambda) = 0.\] 
Since $V_\lambda$ is a vector space over a field of characteristic $0$, the basic properties of the transfer map (see 
\cite[Chapter III.9]{BrownCohomology}) show that the vector space $\HH^{\SLvcd-1}(\GL_n \Z;V_\lambda)$ is a subspace
of $\HH^{\SLvcd-1}(\SL_n \Z;V_\lambda)$, so it is enough to deal with $\SL_n \Z$.  As we discussed
in the introduction, Borel--Serre \cite[Eq. (1)]{BorelSerreCorners} proved that there is an isomorphism
\[\HH^{\SLvcd-k}(\SL_n \Z;V_\lambda)\iso\HH_k(\SL_n \Z; \St_n \otimes V_\lambda) \qquad\quad (k \geq 0)\hskip-20pt.\] 
%TC: note that there is NO DUAL here since one is cohomology and one is homology
%TC: note also that there is no need to require $k\geq 1$ as we were previously (and no reduced homology)

%Removed some coloneq
Let $\BA'_n$ be the subcomplex of $\BA_n$ introduced in \S\ref{section:bykovskiiproofsec}.  
Define $\I_0^{\Q} = C_{n-1}(\BA_n,\BA'_n;\Q)$ and $\I_1^{\Q} = C_n(\BA_n,\BA'_n;\Q)$, and set $\St_n^\Q = \St_n\otimes \Q$. Since $\St_n\otimes V_\lambda=\St_n^\Q\otimes V_\lambda$, our goal is to show that
\begin{equation}
\label{eqn:vanishinggoal}
\HH_1(\SL_n \Z;\St_n^\Q \otimes V_{\lambda}) = 0.
\end{equation}

Since $\Q$ is a flat $\Z$-module, it follows from the proof
of Theorem~\ref{maintheorem:bykovskii} in \S\ref{section:bykovskiiproofsec} that there 
is an exact sequence
\[\I_1^{\Q} \longrightarrow \I_0^{\Q} \longrightarrow \St_n^\Q \longrightarrow 0.\]
The $(n-1)$-simplices of $\BA_n$ that do not lie in $\BA'_n$ are the %complete
frames $\{\l{v_1},\ldots,\l{v_n}\}$ for $\Z^n$. The $\GL_n\Z$-stabilizer of such a frame is a finite group isomorphic to $S_n^{\pm}$, the $2^n\cdot n!$-element group of signed permutation matrices. The $\SL_n\Z$-stabilizer of each frame is thus a subgroup of this finite group, so Lemma~\ref{lemma:recognizeflat} shows that
$\I_0^{\Q}$ is a flat $\SL_n \Z$-module.  Similarly, the $n$-simplices
of $\BA_n$ that do not lie in $\BA'_n$ are the %complete
augmented frames
$\{\l{v_0},\ldots,\l{v_n}\}$ for $\Z^n$. The $\GL_n\Z$-stabilizer of an augmented frame is isomorphic to $D_6\times S_{n-2}^{\pm}$, where $D_6$ is the dihedral group of order 12, so the $\SL_n \Z$-stabilizer of an augmented frame is finite as well.  By Lemma~\ref{lemma:recognizeflat}, $\I_1^{\Q}$ is a flat $\SL_n \Z$-module as well. 

We may therefore extend this exact sequence to 
a flat resolution of the $\SL_n \Z$-module $\St_n^\Q$:
\begin{equation}
\label{eqn:flatresolution}
\cdots \longrightarrow F_3 \longrightarrow F_2 \longrightarrow \I_1^{\Q} \longrightarrow \I_0^{\Q} \longrightarrow \St_n^\Q \longrightarrow 0.
\end{equation}
Lemma~\ref{lemma:flatresolution} says that
$\HH_{\ast}(\SL_n \Z;\St_n^\Q \otimes V_{\lambda})$ is computed by the homology of the
chain complex
\[\cdots \longrightarrow F_3 \otimes_{\SL_n \Z} V_{\lambda} \longrightarrow F_2 \otimes_{\SL_n \Z} V_{\lambda} \longrightarrow \I_1^{\Q} \otimes_{\SL_n \Z} V_{\lambda} \longrightarrow \I_0^{\Q} \otimes_{\SL_n \Z} V_{\lambda} \longrightarrow 0.\]
To prove \eqref{eqn:vanishinggoal}, it is therefore enough to show that
$\I_1^{\Q} \otimes_{\SL_n \Z} V_{\lambda} = 0$ under our assumption that
$n\geq 3+\|\lambda\|$.  We remark that in our earlier paper \cite{CFPSteinberg} with Benson Farb, we  used a similar argument to show that $\I_0^{\Q} \otimes_{\SL_n \Z} V_{\lambda} = 0$ for $n\geq 2+\|\lambda\|$; see \cite[Theorem C]{CFPSteinberg}, which shows the vanishing of $\HH^{\binom{n}{2}}(\SL_n \Z;V_\lambda)$ and applies also to $\SL_n\mathcal{O}_K$ for many number rings $\mathcal{O}_K$.

Define the partition $\lambda' = (\lambda'_1,\ldots,\lambda'_{n-1},0)$ via the formula
$\lambda'_i = \lambda_i - \lambda_n$; observe that $\|\lambda'\|=\|\lambda\|$. 
As  $\GL_n\Q$-representations, we have 
$V_\lambda\iso V_{\lambda'} \tensor \det^{\tensor \lambda_n}$, so as a representation
of $\SL_n\Q$ or $\SL_n\Z$, the representation $V_{\lambda}$ is isomorphic to $V_{\lambda'}$. 
Let $V \coloneq V_{(1)}$ denote the standard $\SL_n \Z$-representation on $\Q^n$.
Using Schur--Weyl duality, we can embed 
$V_{\lambda'}$ as a direct summand of $V^{\otimes k}$ with $k = \|\lambda'\|=\|\lambda\|$.  It thus suffices to show that $\I_1^\Q\otimes_{\SL_n\Z}V^{\otimes k}=0$ when $n \geq 3+k$.\medskip

Fix an augmented frame $\sigma = \{\l{v_0},\ldots,\l{v_n}\}$ for $\Z^n$, and choose representatives so that
$\{v_1,\ldots,v_n\}$ is a basis for $\Z^n$ and $v_0 =v_1 + v_2$. Orienting $\sigma$ using the ordering on the $v_i$ determines a generator $[\sigma]$ of $\I_1^{\Q} \cong C_n(\BA_n,\BA'_n;\Q)$.
Moreover, fix arbitrary indices $i_1,\ldots,i_k\in \{1,\ldots,n\}$ and consider the element $w = v_{i_1} \otimes \cdots \otimes v_{i_k}\in V^{\otimes k}$.

We now prove that the image of $[\sigma] \otimes w$ in
$\I_1^\Q\otimes_{\SL_n\Z}V^{\otimes k}$ is $0$. We do this by constructing an element $\phi$ in the stabilizer of the augmented frame $\sigma$ that satisfies $\phi([\sigma]) = -[\sigma]$ and $\phi(w)=w$. 

Since $n\geq 3+k$, we can find some
$3 \leq j \leq n$ such that $j \notin \{i_1,\ldots,i_k\}$.  We consider two cases separately.
\begin{compactitem}
\item First, if we can find a second index $3\leq j'\leq n$ with $j\neq j'$ and $j'\notin \{i_1,\ldots,i_k\}$, we define $\phi \in \SL_n \Z$ by
\[\phi(v_i) = \begin{cases}
-v_{j'} & \text{if $i = j$},\\
v_j & \text{if $i = j'$},\\
v_i & \text{otherwise}\end{cases} \quad \quad \quad (1 \leq i \leq n).\]
Note that $\phi(v_0)=\phi(v_1+v_2)=v_1+v_2=v_0$, so the element $\phi$ preserves the augmented frame $\sigma$. Since $\phi$ exchanges the lines $\l{v_j}$ and $\l{v_{j'}}$, it reverses the orientation of the corresponding simplex, so $\phi([\sigma])=-[\sigma]$. By construction $\phi$ fixes all the vectors $v_{i_1}, \ldots,v_{i_k}$, so $\phi(w)=w$.
\item The second case is that no such $j'$ exists. Neither $1$ nor $2$ can belong to $\{i_1,\ldots,i_k\}$ since $n\geq 3+k$. 
In this case we define $\phi \in \SL_n \Z$ via the formula
\[\phi(v_i) = \begin{cases}
v_2 & \text{if $i=1$},\\
v_1 & \text{if $i=2$},\\
-v_j & \text{if $i=j$},\\
v_i & \text{otherwise}\end{cases} \quad \quad \quad (1 \leq i \leq n).\]
Note that $\phi(v_0)=\phi(v_1+v_2)=v_2+v_1=v_0$, so again $\phi$ preserves the augmented frame $\sigma$. Since $\phi$ exchanges the lines $\l{v_1}$ and $\l{v_2}$, we have $\phi([\sigma])=-[\sigma]$. Since neither $1$ nor $2$ nor $j$ belongs to $\{i_1,\ldots,i_k\}$, the vectors $v_{i_1}, \ldots,v_{i_k}$ are all fixed by $\phi$, so $\phi(w)=w$.
\end{compactitem}
In both cases, our chosen element satisfies $\phi([\sigma])=-[\sigma]$ and $\phi(w)=w$. It follows that the images of $[\sigma]\otimes w$ and $-[\sigma]\otimes w=\phi([\sigma]\otimes w)$ coincide in $\I_1^\Q\otimes_{\SL_n\Z}V^{\otimes k}$, and thus must be 0.

Since $\{v_1,\ldots,v_k\}$ is a basis for $\Z^n$, the tensors $v_{i_1}\otimes\cdots\otimes v_{i_k}$ constitute a basis for $V^{\otimes k}$. In other words, $w$ was an arbitrary basis element of $V^{\otimes k}$, so the vanishing of $[\sigma]\otimes w$ implies that 
the image of $[\sigma] \otimes V^{\otimes k}$ in
$\I_1^\Q\otimes_{\SL_n\Z}V^{\otimes k}$ vanishes. Since $[\sigma]$ was an arbitrary generator of $\I_1^\Q$, this implies that $\I_1^\Q\otimes_{\SL_n\Z}V^{\otimes k}=0$, as desired. This completes the proof of Theorem~\ref{maintheorem:vanishing}.

\section{The complex of partial augmented frames is CM}
\label{section:proofbasesacon}

The remainder of the paper is occupied with the proof of  Theorem~\ref{maintheorem:basesacon}, which asserts that the $n$-dimensional complex $\BA_n$ is Cohen--Macaulay (CM) of dimension $n$.

\subsection{Warmup: The complex of partial frames is CM}
\label{section:proofbasescon}

Recall from the introduction that $\B_n$ is the complex of partial frames of $\Z^n$.  In this section,
we will prove that $\B_n$ is CM of dimension $(n-1)$.  This theorem is similar to a result of Maazen \cite{MaazenThesis}
and we could deduce it from his work, but we include a proof since
it provides a simpler venue to preview the ideas that we will use in our proof of Theorem~\ref{maintheorem:basesacon}.
Moreover, we will use both this result and the details of its proof in multiple places during the proof of
Theorem~\ref{maintheorem:basesacon}.

During our proof, we will need to understand the links of various simplices of $\B_n$, so we make the following
definition.  Throughout this section, $\{e_1,\ldots,e_p\}$ will denote the standard basis for $\Z^p$; the context
will indicate what value of $p$ we are using at any particular point.

\begin{definition}
\label{def:Bnm}
For $n \geq 0$ and $m \geq 0$, let $\B_n^m$ be the subcomplex
$\Link_{\B_{m+n}}(\{\l{e_1},\ldots,\l{e_m}\})$ of $\B_{m+n}$.
\end{definition}

The main result of this section is the following theorem, which as we said above is closely
related to a theorem of Maazen \cite{MaazenThesis}. Of course, $\B_n^0$ is equal to $\B_n$, so this theorem proves that $\B_n$ is CM of dimension $n-1$, as claimed.

\begin{theorem}
\label{theorem:basescon}
For all $n \geq 0$ and $m\geq 0$, the complex $\PartialBases_n^m$ is CM of dimension $n-1$.
\end{theorem}

We preface the proof of Theorem~\ref{theorem:basescon} with two lemmas.  Analogues of these
two lemmas will be at the heart of our
proof of the more difficult Theorem~\ref{maintheorem:basesacon} (and the second lemma here will also be used directly during that proof).

\begin{lemma}
\label{lemma:connectlinksb}
Consider $n \geq 0$ and $m \geq 0$.  For some $1 \leq k \leq n$, let $\sigma$ be a $(k-1)$-simplex of $\B_n^m$.  Then
the complex $\Link_{\B_n^m}(\sigma)$ is isomorphic to $\B_{n-k}^{m+k}$.
\end{lemma}
\begin{proof}
Write $\sigma = \{\l{v_1},\ldots,\l{v_k}\}$, so $\{e_1,\ldots,e_m,v_1,\ldots,v_k\}$ is a basis for a direct summand
of $\Z^{m+n}$.  Extend this to a basis $\{e_1,\ldots,e_m,v_1,\ldots,v_n\}$ for $\Z^{m+n}$.  Define
$\phi \in \GL_{m+n}(\Z)$ by the formulas $\phi(e_i) = e_i$ for $1 \leq i \leq m$ and $\phi(v_j) = e_{m+j}$ for
$1 \leq j \leq n$.  Then $\phi$ induces an automorphism of $\B_n^m$ that takes $\Link_{\B_n^m}(\sigma)$
to $\B_{n-k}^{m+k} \subset \B_n^m$.
\end{proof}

\begin{definition}
Consider $n \geq 0$ and $m \geq 0$.  Assume that some linear map $F\colon \Z^{m+n} \rightarrow \Z$ has been
fixed. Given a subcomplex $X$ of $\B_n^m$ and given $N>0$, we define $X^{<N}$ to be the
full subcomplex of $X$ spanned by the set of vertices $\l{v}$ of $X$ satisfying
$\abs{F(v)} < N$.  This condition is well-defined since $\abs{F(v)} = \abs{F(-v)}$.
\end{definition}

\begin{lemma}
\label{lemma:bretract} 
Consider $n \geq 0$ and $m \geq 0$.  Let $F\colon \Z^{m+n} \rightarrow \Z$ be a fixed linear map and let $N > 0$.  Let
$\sigma$ be a simplex of $\B_n^m$ such that some vertex $\l{w}$ of $\sigma$ satisfies $F(w) = N$.  Then
there exists a simplicial retraction $\pi\colon \Link_{\B_n^m}(\sigma) \onto \Link_{\B_n^m}(\sigma)^{<N}$.
\end{lemma}
\begin{proof}
Define $X = \Link_{\B_n^m}(\sigma)$ and write $\sigma = \{\l{w_1},\ldots,\l{w_p}\}$ with $w_1=w$. Our goal is to construct a simplicial retraction $X\onto X^{<N}$.  Say
that $v \in \Z^{m+n}$ is \emph{$F$-nonnegative} if $F(v) \geq 0$.
We begin by defining a map $\widehat{\pi}\colon X^{(0)} \rightarrow (X^{<N})^{(0)}$ on $0$-simplices as follows.
\begin{compactitem}
\item Consider a vertex $\l{v}$ of $X$.  Replacing $v$ with $-v$ if necessary, we can assume that
$v$ is $F$-nonnegative.  Define $q_v \in \N$ to be the result $\lfloor \frac{F(v)}{N} \rfloor$
of dividing $F(v)$ by $N$, so $0 \leq F(v) - q_v N < N$.  We then set
\begin{equation}
\label{eq:pihat}
\widehat{\pi}(\l{v}) = \l{(v - q_v w)}.
\end{equation}
This is well-defined; the only possible ambiguity occurs when $F(v) = 0$ and hence both
$v$ and $-v$ are $F$-nonnegative, but in that case we have $q_v=q_{-v}=0$ so $\widehat{\pi}(\l{v}) = \l{v}$ no
matter what choice we make.
\end{compactitem}
By definition,
$\widehat{\pi}(\l{v}) = \l{v}$ if $\l{v}$ is a vertex of $X^{<N}$, and similarly $\widehat{\pi}(\l{v}) \in X^{<N}$ for any vertex $\l{v}$ of $X$. To complete the proof, we must prove that
$\widehat{\pi}$ extends over the higher-dimensional simplices of $X$.  Consider a $(k-1)$-simplex
$\{\l{v_1},\ldots,\l{v_k}\}$ of $X$, so $\{e_1,\ldots,e_m,w_1,\ldots,w_p,v_1,\ldots,v_k\}$ is a 
basis for a rank-$(m+p+k)$
direct summand $U$ of $\Z^n$.  Replace  the $v_i$ by $-v_i$ if necessary to ensure that the $v_i$ are $F$-nonnegative and set $v'_i = v_i - q_{v_i} w_1$,
so $\widehat{\pi}(\l{v_i}) = \l{(v'_i)}$.  Since each $v'_i$ is obtained by adding some multiple of $w_1$ to $v_i$,
the set $\{e_1,\ldots,e_m,w_1,\ldots,w_p,v'_1,\ldots,v_k'\}$ 
is also a basis for $U$.
We conclude
that $\{\widehat{\pi}(\l{v_1}),\ldots,\widehat{\pi}(\l{v_k})\}$ is a $(k-1)$-simplex of $X^{<N}$, as desired.
\end{proof}

\begin{remark}
\label{remark:summandofsummand}
If a partial frame $\sigma$ of $\Z^\ell$ is contained in a summand $V$ of $\Z^\ell$, then in fact $\sigma$ is a partial frame of $V$. Although this fact may seem obvious, it need not hold over other rings and its failure can lead to great difficulty. For example, over the ring $A=\mathbb{R}[x,y,z]/(x^2+y^2+z^2-1)$,
the vector $v=xe_1+ye_2+ze_3$ is part of a basis for $A^4$ and has $e_4$-coordinate 0, but $v$ is \emph{not} part of any basis containing the vector $e_4$. 
% TC: here is an alternate way to write this sentence. I'm neutral between them, but I do want to insist that we include this explicit example of why the claim requires proof.
% the vector $v=[x,y,z,0]$ is part of a basis for $A^4$, but is \emph{not} part of any basis containing the vector $e_4=[0,0,0,1]$.
Nevertheless, over $\Z$ the claim follows from the fact that if a summand $U$ of $\Z^\ell$ is contained in another summand $V$ of $\Z^\ell$, then $U$ is a summand of $V$; this property holds not only for $\Z$ but for any Dedekind domain. For the same reason, a partial augmented frame of $\Z^\ell$ that is contained in a summand $V$ of $\Z^\ell$ is in fact a partial augmented frame of $V$; this will be used in the next section in the proof of Theorem~\ref{maintheorem:basesacon}.
\end{remark}

We now come to the proof of Theorem~\ref{theorem:basescon}.  This proof could be written in the language
of combinatorial Morse theory without great difficulty, but it would be much more awkward to express our later proof of Theorem
\ref{maintheorem:basesacon} in this language (as we illustrate afterwards in Remark~\ref{remark:PLMorsedifficulty}).  Since our goal is to motivate the proof of
Theorem~\ref{maintheorem:basesacon}, we follow its structure here.

\begin{proof}[Proof of Theorem~\ref{theorem:basescon}]
We prove the theorem by induction on $n$.  For the base case $n=0$, we must prove for all $m \geq 0$
that $\B_0^m$ is CM of dimension $-1$, i.e.\ that the simplicial complex  $\B_0^m$ is empty.  Since $\{\l{e_1},\ldots,\l{e_m}\}$ is already a 
frame for $\Z^{m}$, it is a maximal simplex of $\B_m$. Therefore its link $\B_0^m$ is empty, as desired.

Now fix $n>0$ and $m \geq 0$ and assume that $\B_{n'}^{m'}$ is CM of dimension $n'-1$ for all $n'<n$ and all $m'\geq 0$. 
Since every frame for $\Z^{m+n}$ consists of $m+n$ lines, the complex $\PartialBases_n^m$ is $(n-1)$-dimensional.
Lemma~\ref{lemma:connectlinksb} and our induction hypothesis implies that for all $(k-1)$-simplices $\sigma$
of $\B_n^m$ with $1 \leq k \leq n$, the complex $\Link_{\B_n^m}(\sigma)$ is CM of dimension $n-k$.  All
that remains to show is that $\B_n^m$ is $(n-2)$-connected.

Fix $0 \leq p \leq n-2$, let $S^p$ be a combinatorial triangulation of a $p$-sphere, and let 
$\phi\colon S^p \rightarrow \B_{n}^{m}$ be a simplicial map.  Our goal is to show that $\phi$ 
can be homotoped to a constant map.
Let $F\colon \Z^{m+n} \onto \Z$ be the linear map taking $v \in \Z^{m+n}$ to the $e_{m+n}$-coordinate of $v$.
For a vertex $\l{v}$ of $\B_n^m$, define $\rank(\l{v}) = \abs{F(v)}$; this is well-defined since
$\abs{F(v)} = \abs{F(-v)}$.  We then define
\[R(\phi) = \Max \Set{$\rank(\phi(x))$}{$x$ a vertex of $S^p$}.\]
This will be our measure of complexity for $\phi$.

If $R(\phi)=0$, then every simplex $\sigma$ of $\phi(S^p)$ is contained in the summand $\ker F$ of $\Z^{m+n}$. In particular, $\{\l{e_1},\ldots,\l{e_m}\}\ast\sigma$ is a partial frame contained in $\ker F$; by Remark~\ref{remark:summandofsummand}, it is in fact a partial frame for $\ker F$, so it can be extended to a partial frame for $\Z^{n+m}$ by adding the line $\l{e_{m+n}}$. In other words, the entire image $\phi(S^p)$ is contained in the star (indeed, in the link) of $\l{e_{m+n}}$. We conclude that when $R(\phi)=0$, the desired null-homotopy is obtained by homotoping $\phi$ to the constant map at the vertex $\l{e_{m+n}}$.

We can therefore assume that $R(\phi)=R>0$; we want to homotope $\phi$ so as to reduce $R(\phi)$.
Consider the following condition on a simplex $\sigma$ of $S^p$:
\begin{equation}
\label{eq:Bcond}
\rank(\phi(x)) = R\text{ for all vertices }x\text{ of }\sigma.
\end{equation}
Since $R(\phi) = R$, there must be some simplex $\sigma$ of $S^p$ satisfying \eqref{eq:Bcond}.  We can therefore
choose a simplex $\sigma$ of $S^p$ satisfying \eqref{eq:Bcond} whose dimension $k$ is maximal among those satisfying
\eqref{eq:Bcond}.  This maximality implies that $\phi$ takes $\Link_{S^p}(\sigma)$ to
$\Link_{\B_n^m}(\phi(\sigma))^{<R}$.  

Let $\ell$ be the dimension
of the simplex $\phi(\sigma)$; we certainly have $\ell \leq k$, but we might have $\ell < k$ if $\phi$ restricted to
$\sigma$ is not injective.  Combining Lemma~\ref{lemma:connectlinksb} with our induction
hypothesis, we see that $\Link_{\B_n^m}(\phi(\sigma))$ is CM of dimension $(n-\ell-2)$,
and in particular is $(n-\ell-3)$-connected. This retracts to $\Link_{\B_n^m}(\phi(\sigma))^{<R}$ by
Lemma~\ref{lemma:bretract}, so its retract $\Link_{\B_n^m}(\phi(\sigma))^{<R}$ is also $(n-\ell-3)$-connected.

By the definition of a combinatorial triangulation, the link
$\Link_{S^p}(\sigma)$ is a combinatorial $(p-k-1)$-sphere.    Since $p \leq n-2$ and $\ell \leq k$, we have $p-k-1 \leq n-\ell-3$, so $\phi|_{\Link_{S^p}(\sigma)}$ is null-homotopic via a homotopy inside $\Link_{\B_n^m}(\phi(\sigma))^{<R}$.
Using Zeeman's relative simplicial approximation theorem \cite{Zeeman}, we
conclude that there exists a combinatorial $(p-k)$-ball $B$ with $\partial B \iso \Link_{S^p}(\sigma)$
and a simplicial map $\psi\colon B \rightarrow \Link_{\B_n^m}(\phi(\sigma))^{<R}$ such that 
$\psi|_{\partial B} = \phi|_{\Link_{S^p}(\sigma)}$.

The map $\psi$ extends to the $(p+1)$-ball $\sigma\ast B$ as $(\phi|_\sigma) \ast \psi \colon \sigma\ast B\to \B_n^m$. 
The boundary of $\sigma\ast B$ is the union of the $p$-ball $\sigma\ast(\partial B)=\Star_{S^p}(\sigma)$, on which 
$\phi|_\sigma\ast \psi=\phi|_{\Star_{S^p}(\sigma)}$, and the $p$-ball $(\partial \sigma)\ast B$.  We can thus homotope 
$\phi$ across this $(p+1)$-ball to replace $\phi|_{\Star_{S^p}(\sigma)}$ with 
$\phi|_{\partial \sigma}\ast\psi\colon (\partial \sigma)\ast B\to \B_n^m$. 

The key property of this modification is that it eliminates the simplex $\sigma$ and does not add any 
other simplices satisfying \eqref{eq:Bcond}. Indeed, every new simplex is the join of a simplex in $\partial\sigma$ with a nonempty simplex in $B$; since $\psi(B)$ is contained in $\Link_{\B_n^m}(\phi(\sigma))^{<R}$, such a simplex has at least one vertex with $r(\phi(x))<R$, so it will not satisfy \eqref{eq:Bcond}.
Repeating this process, we can homotope $\phi$ to eliminate \emph{all} simplices satisfying 
\eqref{eq:Bcond}; in other words, we can homotope $\phi$ so that $R(\phi)<R$.

By induction, we can homotope $\phi$ so that $R(\phi)=0$. At this point, as explained above, $\phi$ can be directly contracted to a constant map, so this concludes the proof that $\B_n^m$ is $(n-2)$-connected.
\end{proof}

\subsection{The complex \texorpdfstring{$\BA_n^m$}{BA(n,m)}}
\label{section:introducebanm}

We now turn to the proof of Theorem~\ref{maintheorem:basesacon}, which asserts that the complex
$\BA_n$ is CM of dimension $n$.  
Just as for $\B_n$, we will need to understand links of simplices in $\BA_n$.  However, for
technical reasons the heart of our argument will deal not with the entire link, but rather with the following subcomplex of the link.
Recall that $\{e_1,\ldots,e_p\}$ denotes the standard basis for $\Z^p$, where $p \geq 1$ is
determined by context.

\begin{definition}
\label{def:BAmn}
For $n \geq 1$ and $m \geq 0$ with $m+n\geq 2$, define $\BA_n^m$ to be the full subcomplex of
 $\Link_{\BA_{m+n}}(\{\l{e_1},\ldots,\l{e_m}\})$ spanned by vertices $\l{v}$ of
$\Link_{\BA_{m+n}}(\{\l{e_1},\ldots,\l{e_m}\})$ such that $v \notin \Span_{\Z}(e_1,\ldots,e_m) \subset \Z^{m+n}$.
\end{definition}

For example, even though $\{\l{e_1},\l{e_2},\l{(e_1+e_2)}\}$ is a simplex
of $\BA_{n+2}$, the vertex $\l{(e_1+e_2)}$ is excluded from $\BA_n^2$.
Our main theorem is then as follows.  It reduces to Theorem~\ref{maintheorem:basesacon} when $m=0$.

\begin{theorem-prime}{maintheorem:basesacon}
\label{theorem:basesacon}
For $n \geq 1$ and $m \geq 0$ with $m+n\geq 2$, the complex $\BA_n^m$ is CM of dimension~$n$.
\end{theorem-prime}

\begin{remark}
We have intentionally refrained from defining $\BA_n^m$ in the case when $n=0$ or the case when $m+n<2$. The reason is that $\BA_n^m$ would be degenerate in these cases; not only would Theorem~\ref{theorem:basesacon} be false in these cases, $\BA_n^m$ would not even be $n$-dimensional.
\end{remark}

We will prove Theorem~\ref{theorem:basesacon} in \S\ref{section:basesacon2proof}.  This is
preceded by \S\ref{section:basesacon2links}, which describes the links in $\BA_n^m$ (or certain subcomplexes of the links) and establishes the base case for our induction, and by \S\ref{section:basesacon2retract}, which constructs certain
retractions on links in $\BA_n^m$.  Before we start
with all of this, we close this section by introducing some terminology for simplices of $\BA_n^m$.

\begin{definition}
\label{def:trichotomy}
Fix $n \geq 1$ and $m \geq 0$ with $m+n\geq 2$.
We divide the simplices of $\BA_n^m$ into three mutually exclusive types.
\begin{compactitem}
\item A \emph{standard simplex} is a simplex $\{\l{v_1},\ldots,\l{v_p}\}$ such that
$\{\l{e_1},\ldots,\l{e_m},\l{v_1},\ldots,\l{v_p}\}$ is a simplex of $\B_n^m$.  In other
words, $\{e_1,\ldots,e_m,v_1,\ldots,v_p\}$ is a basis for a direct summand of $\Z^{m+n}$.
\item An \emph{internally additive simplex} is a simplex that can be written as
$\{\l{v_0},\ldots,\l{v_p}\}$, where $\{\l{v_1},\ldots,\l{v_p}\}$ is a standard simplex
and $\pm v_0 \pm v_1 \pm v_2 = 0$ for some choice of signs.  We will call $\{\l{v_0}, \l{v_1}, \l{v_2}\}$
the \emph{additive core} of our simplex; this subset is well-defined since $\{v_0,v_1,v_2\}$ is the minimal
 linearly dependent subset of $\{v_0,\ldots,v_p\}$.
\item An \emph{externally additive simplex} is a simplex that can be written as
$\{\l{v_0},\ldots,\l{v_p}\}$, where $\{\l{v_1},\ldots,\l{v_p}\}$ is a standard simplex
and $\pm v_0 \pm v_1 \pm e_i = 0$ for some choice of signs and some $1 \leq i \leq m$.
We will call $\{\l{v_0},\l{v_1}\}$ the \emph{additive core} of our simplex; it is well-defined
just as for internally additive simplices.
\end{compactitem}
An \emph{additive simplex} is a simplex which is either internally or externally additive.
\end{definition}
\begin{remark}
We emphasize that the classification in Definition~\ref{def:trichotomy} applies to a simplex \emph{as a simplex of $\BA_n^m$}. The same collection of lines might be classified differently as a simplex of $\BA_{n'}^{m'}$. For example, a partial frame that forms an externally additive simplex of $\BA_n^m$ would be a standard simplex when considered as a simplex of $\BA_{n+m}^{0}$.
\end{remark}

\subsection{Describing links in \texorpdfstring{$\BA_n^m$}{BA(n,m)}}
\label{section:basesacon2links}

%%I commented this out because it is not necessary anywhere 
%\begin{lemma}
%\label{lemma:dimension}
%Consider $n \geq 1$ and $m \geq 0$ with $m+n \geq 2$. Every simplex $\sigma=\{\l{v_1},\ldots,\l{v_k}\}$ of $\BA_n^m$ is contained
%in an $n$-dimensional simplex.
%\end{lemma}
%\begin{proof}
%If $\sigma$ is a standard simplex, then $\{\l{e_1},\ldots,\l{e_m},\l{v_1},\ldots,\l{v_k}\}$ is a frame
%for a direct summand of $\Z^{m+n}$, so we can extend it to a frame $\{\l{e_1},\ldots,\l{e_m},\l{v_1},\ldots,\l{v_n}\}$
%for $\Z^{m+n}$.  If $m>0$, set $v_0= e_1+v_1$; then $\{\l{v_0},\l{v_1},\ldots,\l{v_n}\}$ is an externally additive $n$-simplex
%of $\BA_n^m$.  If $m=0$, then since $m+n \geq 2$ we must have $n>1$ and hence we can set $v_0 = v_1+v_2$ and get
%an internally additive $n$-simplex $\{\l{v_0},\l{v_1},\ldots,\l{v_n}\}$ of $\BA_n^m$.
%
%If $\sigma$ is an additive simplex, then 
%$\{\l{e_1},\ldots,\l{e_m},\l{v_1},\ldots,\l{v_k}\}$ is an
%augmented frame for a direct summand of $\Z^{m+n}$. Adjoining a frame for a complement, we can
%extend it to an augmented frame $\{\l{e_1},\ldots,\l{e_m},\l{v_1},\ldots,\l{v_{n+1}}\}$ for $\Z^{m+n}$.
%Then $\{\l{v_1},\ldots,\l{v_{n+1}}\}$ is an additive $n$-simplex of $\BA_n^m$.  
%\end{proof}

In this section, we describe the links of simplices in $\BA_n^m$, as we did for $\B_n^m$ in Lemma~\ref{lemma:connectlinksb}. To handle the link of a standard simplex, we are forced to deal with a certain subcomplex of the link (just as  $\BA_n^m$ is a \emph{subcomplex} of the full link in $\BA_{m+n}$); the reason is that the retraction constructed in Proposition~\ref{prop:baretract2} below cannot be extended across the entire link.

%%%% NOTE TO ANDY: 
% this has been redefined so that \PLink is ONLY DEFINED for standard simplices
\begin{definition}
\label{def:PLink}
Given a standard simplex $\sigma = \{\l{v_1},\ldots,\l{v_p}\}$ of $\BA_n^m$, define
$\PLink_{\BA_n^m}(\sigma)$ to be the full subcomplex of $\Link_{\BA_n^m}(\sigma)$ spanned
by vertices $\l{v}$ of $\Link_{\BA_n^m}(\sigma)$ such that 
$v \notin \Span_{\Z}(e_1,\ldots,e_m,v_1,\ldots,v_p) \subset \Z^{m+n}$.
\end{definition}

\begin{lemma}
\label{lemma:connectlinksba}
Consider $n \geq 1$ and $m \geq 0$ with $m+n\geq 2$.  For some $1 \leq k \leq n+1$, let $\sigma$ be a $(k-1)$-simplex of $\BA_n^m$.\begin{compactenum}[label={\normalfont (\alph*)},ref={(\alph*)}]
\item\label{part:a} If $\sigma$ is an additive simplex, then $\Link_{\BA_n^m}(\sigma)$ is isomorphic to $\B_{n-k+1}^{m+k-1}$.
\item\label{part:b} If $\sigma$ is a standard simplex and $k\neq n$, then $\PLink_{\BA_n^m}(\sigma)$ is isomorphic to
$\BA_{n-k}^{m+k}$.
\item\label{part:c} If $\sigma$ is a standard simplex, define $X = \Link_{\BA_n^m}(\sigma)$ and $\widehat{X} = \PLink_{\BA_n^m}(\sigma)$. 
Then for all vertices $\l{v}$ of $X$ that do not lie in $\widehat{X}$, the complex
$\Link_X(\l{v})$ lies in $\widehat{X}$ and is isomorphic to $\B_{n-k}^{m+k}$. 
\end{compactenum}
\end{lemma}
\begin{proof}
Parts~\ref{part:a} and \ref{part:b} are proved exactly like Lemma~\ref{lemma:connectlinksb} since by an appropriate automorphism of $\BA_n^m$, we may assume that $\sigma=\{\l{e_{m+1}},\ldots,\l{e_{m+k-1}},\l{(e_{m+1}+e_{m+2})}\}$ or $\sigma=\{\l{e_{m+1}},\ldots,\l{e_{m+k}}\}$, respectively. Part~\ref{part:c} is a consequence of Part~\ref{part:a}.
\end{proof}

In the proof of the next proposition, we make use of the following lemma. It is certainly standard, but we could not find a proof in the literature.

\begin{lemma}
\label{lemma:coneCM}
Let $X$ be obtained from the simplicial complex $Y$ by coning off the subcomplex $Z$. If $Y$ is CM of dimension~$n$ and $Z$ is CM of dimension~$n-1$, then $X$ is CM of dimension~$n$.
\end{lemma}
\begin{proof}
Let $p$ be the cone point.  Since $Z$ is $(n-2)$-connected, the pair $(X,Y)$ is $(n-1)$-connected, so $X$ is $n$-spherical. 
There are four kinds of simplices of $X$.  The first is $\{p\}$, whose link is $\Link_X \{p\} = Z$,
which is CM of dimension $(n-1)$ by assumption.  The second is $\sigma \ast \{p\}$ 
for some simplex $\sigma$ of $Z$; its link is $\Link_X (\sigma \ast \{p\}) = \Link_Z \sigma$, 
which is CM of the appropriate dimension
since $Z$ is CM of dimension $(n-1)$.  The third is a simplex $\sigma$ of $Y$ that does not lie in $Z$; its link is $\Link_X \sigma = \Link_Y \sigma$, which is CM of the appropriate dimension since $Y$ is CM of dimension $n$.
The fourth is a simplex $\sigma$ of $Z$; its link is $\Link_X \sigma = \{p\} \ast \Link_Z \sigma$, which is CM 
of the appropriate dimension because $Z$ is CM of dimension $(n-1)$.
\end{proof}

\begin{proposition}
\label{prop:makecm}
Fix $n \geq 1$ and $m \geq 0$ such that $m+n \geq 2$.  Assume that $\BA_{n'}^{m'}$ is CM of dimension $n'$ for
all $1 \leq n' < n$ and $m' \geq 0$ such that $m'+n' = m+n$.  Then for every $(k-1)$-simplex
$\sigma$ of $\BA_n^m$, the subcomplex $\Link_{\BA_n^m}(\sigma)$ is CM of dimension $n-k$.
\end{proposition}
\begin{proof}
If $\sigma$ is an additive simplex, then Lemma~\ref{lemma:connectlinksba}\ref{part:a} asserts that $\Link_{\BA_n^m}(\sigma)$ is isomorphic to $\B_{n-k+1}^{m+k-1}$, which Theorem~\ref{theorem:basescon} says is CM of dimension $n-k$.

If $\sigma$ is a standard simplex with $k=n$, then we can write $\sigma=\{\l{v_1},\ldots,\l{v_n}\}$
with $\{\l{e_1},\ldots,\l{e_m},\l{v_1},\ldots,\l{v_n}\}$ a frame for $\Z^{m+n}$.
If $m>0$, set $v_0= e_1+v_1$; otherwise, since $m+n \geq 2$ we must have $n\geq 2$ and we can set $v_0 = v_1+v_2$. 
In either case, the set $\{\l{e_1},\ldots,\l{e_m},\l{v_0},\l{v_1},\ldots,\l{v_n}\}$ is an augmented frame for 
$\Z^{m+n}$, so $\l{v_0}$ is contained in $\Link_{\BA_n^m}(\sigma)$. 
Therefore $\Link_{\BA_n^m}(\sigma)$ is nonempty, i.e.\ CM of dimension $0$.

Finally, if $\sigma$ is a standard simplex with $k<n$, define 
$X=\Link_{\BA_n^m}(\sigma)$ and $\widehat{X} = \PLink_{\BA_n^m}(\sigma)$.  Since $1\leq k<n$,  
Lemma~\ref{lemma:connectlinksba}\ref{part:b} asserts that $\widehat{X}$ is isomorphic to $\BA_{n-k}^{m+k}$, 
which by assumption is CM of dimension $n-k$.  Let $\l{v}$ be a vertex of $X$ that does not lie in
$\widehat{X}$.  Lemma~\ref{lemma:connectlinksba}\ref{part:c} says that 
$\Link_X(\l{v})$ lies in $\widehat{X}$ and is isomorphic to $\B_{n-k}^{m+k}$, so adding $\l{v}$ to $\widehat{X}$ has the
effect of coning off the subcomplex $\Link_X(\l{v}) \cong \B_{n-k}^{m+k}$.  This subcomplex 
is CM of dimension $n-k-1$ by Theorem~\ref{theorem:basescon}, so Lemma~\ref{lemma:coneCM} tells us that  coning off this subcomplex preserves the property of being CM of dimension $n-k$. Carrying this out for each vertex of $X$ not contained in $\widehat{X}$, we conclude that $X=\Link_{\BA_n^m}(\sigma)$ is CM of dimension $n-k$, as desired.
\end{proof}

\subsection{The retraction maps}
\label{section:basesacon2retract}

In this section, we construct two retractions of the links in $\BA_n^m$ (or parts of them), just as we did for $\B_n^m$ in  Lemma~\ref{lemma:bretract}.  We begin with the following definition.

\begin{definition}
%Consider $n \geq 1$ and $m \geq 0$ with $m+n\geq 2$.
Assume that some linear map 
$F\colon \Z^{m+n} \rightarrow \Z$ has been
fixed and let $X$ be a subcomplex of $\BA_n^m$.  For $N>0$, we define $X^{<N}$ to be the
full subcomplex of $X$ spanned by the set of vertices $\l{v}$ of $X$ satisfying
$\abs{F(v)} < N$.  This is well-defined since $\abs{F(v)} = \abs{F(-v)}$.  
\end{definition}

Our first retraction, for the link of an additive simplex, is straightforward.
\begin{lemma}
\label{lemma:baretract1}
Consider $n\geq 2$ and $m\geq 0$.
%Consider $n \geq 1$ and $m \geq 0$ with $m+n\geq 2$. 
Let $F\colon \Z^{m+n} \rightarrow \Z$ be a fixed
linear map.  Let $\sigma$ be an additive 
simplex of $\BA_n^m$ such that there
exists some vertex $\l{w}$ of $\sigma$ with $F(w) = N>0$.  Then
there exists a simplicial retraction
$\pi\colon \Link_{\BA_n^m}(\sigma) \onto \Link_{\BA_n^m}(\sigma)^{<N}$.
\end{lemma}
\begin{proof}
The retraction $\pi$ is defined on vertices by the same formula \eqref{eq:pihat} as in Lemma~\ref{lemma:bretract}.  
The fact that $\sigma$ is an additive simplex ensures
that for all simplices $\tau$ of $\Link_{\BA_n^m}(\sigma)$, the additive core
of $\sigma \ast \tau$ is disjoint from $\tau$, which implies that there is no difficulty in extending $\pi$
over $\tau$.
\end{proof}

Our second retraction will be more difficult to construct because internally additive simplices are extremely constrained. Indeed, if two lines $\l{v_1}$ and $\l{v_2}$ are specified, there are only \emph{two} lines $\l{v_0}$ for which $\{\l{v_0},\l{v_1},\l{v_2}\}$ is an internally additive simplex.
As a result, if we attempt to define a retraction on the link of a standard simplex by \eqref{eq:pihat} as in Lemma~\ref{lemma:baretract1}, the retraction will  not extend across all additive simplices.

For example, consider the vectors $v_1=e_1+9e_4$, $v_2=e_2+9e_4$, $v_0=v_1+v_2=e_1+e_2+18e_4$, and $w=e_3+10e_4$, so $\{\l{v_0},\l{v_1},\l{v_2}\}$ forms an additive simplex of $\Link_{\BA_4}(\l{w})$. However, if we take $F\colon \Z^n\to \Z$ to be the coefficient of $e_4$ and define $\widehat{\pi}$ as in \eqref{eq:pihat}, then we have $\widehat{\pi}(\l{v_1})=\l{v_1}$ and $\widehat{\pi}(\l{v_2})=\l{v_2}$, but $\widehat{\pi}(\l{v_0})=\l{(v_0-w)}=\l{(e_1+e_2-e_3+8e_4)}$. Thus $\{\widehat{\pi}(\l{v_0}), \widehat{\pi}(\l{v_1}), \widehat{\pi}(\l{v_2})\}$ is not a simplex of $\Link_{\BA_n}(\l{w})$ at all.

In general this problem seems insuperable. We will solve it only for the link of a single vertex and only after restricting to the subcomplex $\PLink_{\BA_n^m}(\l{w})$; even then, to make the retraction well-defined we are forced to subdivide the complex first. This is the content of the following proposition.

\begin{proposition}
\label{prop:baretract2}
Consider $n\geq 2$ and $m\geq 0$.
%Consider $n \geq 1$ and $m \geq 0$ with $m+n\geq 2$. 
Let $F\colon \Z^{m+n} \rightarrow \Z$ be a fixed linear map such that $F(e_i) = 0$ for $1 \leq i \leq m$. Let
$\l{w}$ be a vertex of $\BA_n^m$ such that $F(w) = N>0$.    Then there exists a topological retraction $\pi\colon \PLink_{\BA_n^m}(\l{w}) \onto \PLink_{\BA_n^m}(\l{w})^{<N}$.
\end{proposition}
\begin{proof}
%Our construction of $\pi$ will initially parallel the construction in the proof of Lemma~\ref{lemma:bretract}. 
Define
$X = \PLink_{\BA_n^m}(\l{w})$, so our goal is to construct a topological retraction $\pi\colon X\onto X^{<R}$. 
We begin by defining a map $\widehat{\pi}\colon X^{(0)} \rightarrow (X^{<N})^{(0)}$ on $0$-simplices by the same formula \eqref{eq:pihat} as before.
To recap, we say that $v \in \Z^{m+n}$ is \emph{$F$-nonnegative} if $F(v) \geq 0$. For $F$-nonnegative $v$ we define $q_v = \lfloor \frac{F(v)}{N} \rfloor\in \N$ and set $\widehat{\pi}(\l{v}) = \l{(v - q_v w)}$, so $0\leq F(\widehat{\pi}(\l{v}))<N$.
%Removed some coloneq

If $\l{v}$ is a vertex of $X$, then by the definition of $\PLink_{\BA_n^m}(\l{w})$ we know that $v\notin \Span_{\Z}(e_1,\ldots,e_m,w)$, so $\{\l{e_1},\ldots,\l{e_m},\l{w},\l{v}\}$ is a partial frame for a rank $m+2$ summand of $\Z^{m+n}$. Thus $\{\l{e_1},\ldots,\l{e_m},\l{w},\widehat{\pi}(\l{v})\}$ is a partial frame for the same summand, so $\widehat{\pi}(\l{v})$ is a vertex of $X^{<N}$; moreover $\widehat{\pi}(\l{v}) = \l{v}$ if $\l{v}$ is a vertex of $X^{<N}$. In other words, $\widehat{\pi}$ is a retraction of the vertices of $X$ onto the vertices of $X^{<N}$.

Unfortunately, the map $\widehat{\pi}$ does not extend to a simplicial map on $X$, as we discussed above.
 What we will show
instead is that there exists a subdivision $Y$ of $X$ such that $X^{<N}$ is still a subcomplex of $Y$ (so
no simplices of $X^{<N}$ are subdivided) and an extension of $\widehat{\pi}$ to $Y$.

The trouble will occur only on the internally additive simplices.  Before we deal with these, we prove that
$\widehat{\pi}$ extends over the other simplices of $X$.
We distinguish the standard simplices lying in $X$ into two types:
\begin{compactitem}
\item A \emph{$\l{w}$-standard simplex} is a simplex $\sigma$ of $X$ such that $\sigma \ast \{\l{w}\}$ is a
standard simplex of $\BA_n^m$.
\item A \emph{$\l{w}$-additive simplex} is a simplex of $X$ that can be written in the form
$\{\l{v_0},\ldots,\l{v_p}\}$ with $\{\l{v_1},\ldots,\l{v_p}\}$ a  standard simplex of $\BA_n^m$ and
$\pm v_0 \pm v_1 \pm w = 0$ for some choice of signs.\end{compactitem}
%\BeginClaims
\begin{claims}
\label{claim:wstandard}
The map $\widehat{\pi}$ extends over the $\l{w}$-standard simplices $\sigma$ of $X$.
\end{claims}
\begin{proof}[Proof of claim]
This is identical to the proof of the corresponding statement in the proof of Lemma~\ref{lemma:bretract}.
\end{proof}

\begin{claims}
The map $\widehat{\pi}$ extends over the externally additive simplices $\sigma$ of $X$.
\end{claims}
\begin{proof}[Proof of claim]
Write $\sigma=\{\l{v_0},\ldots,\l{v_p}\}$, where each $v_i$ is $F$-nonnegative, $\{\l{v_1},\ldots,\l{v_p}\}$ is a $\l{w}$-standard simplex, and $\pm v_0\pm v_1\pm e_i=0$ for some $i$ and some choice of signs. 
Since $F(e_i)=0$ and both $F(v_0)$ and $F(v_1)$ are nonnegative, the relation $\pm v_0\pm v_1\pm e_i=0$ implies that $F(v_0) = F(v_1)$.  Moreover, possibly replacing $v_1$ by $-v_1$ if $F(v_1)=0$, we have $v_0 = v_1 + \sign e_i$ for some $\sign \in \{\pm1\}$.   Since $F(v_0)=F(v_1)$ we have $q_{v_0} = q_{v_1}$, so
\[\{\widehat{\pi}(\l{v_0}),\ldots,\widehat{\pi}(\l{v_p})\} = \{\l{(v_1 - q_{v_1} w+ \sign e_i )},\l{(v_1-q_{v_1} w)},\ldots,\l{(v_p - q_{v_p} w)}\}.\]
This is an externally additive simplex of $X^{<N}$.
\end{proof}

\begin{claims}
The map $\widehat{\pi}$ extends over the $\l{w}$-additive simplices $\sigma$ of $X$.
\end{claims}
\begin{proof}[Proof of claim]
Write $\sigma=\{\l{v_0},\ldots,\l{v_p}\}$, where each $v_i$ is $F$-nonnegative, $\{\l{v_1},\ldots,\l{v_p}\}$ is a $\l{w}$-standard simplex, and $\pm v_0\pm v_1\pm w =0$ for some choice of signs.  
Exchanging $v_0$ and $v_1$ if necessary, we can assume that $F(v_0)\geq F(v_1)$. 

We first consider the case where $F(v_0)\geq N$.
Since $F(w)=N$, in this case the relation $\pm v_0\pm v_1\pm w=0$ implies that $v_0=v_1+w$ as long as $F(v_1)>0$. When $F(v_1)=0$, it implies only that $v_0=\pm v_1+w$, but replacing $v_1$ by $-v_1$ we can still assume that $v_0=v_1+w$. The relation $v_0=v_1+w$ implies that $q_{v_0}=q_{v_1}+1$, so we have 
\[\pihat(\l{v_0})=\l{(v_0-q_{v_0}w)}=\l{((v_1+w)-(q_{v_1}+1)w)}=\l{(v_1-q_{v_1}w)}=\pihat(\l{v_1}).\]
In other words, the $\l{w}$-additive edge $\{\l{v_0},\l{v_1}\}$ of $X$ is collapsed by $\widehat{\pi}$ to a single vertex of $X^{<N}$. Similarly, the $\l{w}$-additive $p$-simplex $\sigma$ is collapsed by $\widehat{\pi}$ to
\[\{\widehat{\pi}(\l{v_0}),\ldots,\widehat{\pi}(\l{v_p})\} =\{\l{(v_1-q_{v_1} w)},\ldots,\l{(v_p - q_{v_p} w)}\},\]
a $\l{w}$-standard $(p-1)$-simplex of $X^{<N}$.  We remark that this is the only case in which the dimension of a simplex is decreased by $\widehat{\pi}$.

In the remaining case, we have $0 \leq F(v_1)\leq F(v_0)<N$. Since $F(w)=N$, the relation $\pm v_0\pm v_1\pm w=0$ implies in this case that $w=v_0+v_1$. Since $F(v_0)<N$ and $F(v_1)<N$, we have $q_{v_0} = q_{v_1} = 0$, so $\pihat(\l{v_0})=\l{v_0}$ and $\pihat(\l{v_1})=\l{v_1}$. Therefore 
\begin{align*}
\{\widehat{\pi}(\l{v_0}),\pihat(\l{v_1}),\ldots,\widehat{\pi}(\l{v_p})\} 
&= \{\l{v_0}, \l{v_1}, \l{(v_2 - q_{v_2} w)},\ldots,\l{(v_p - q_{v_p} w)}\}.
\end{align*}
is a $\l{w}$-additive $p$-simplex of $X^{<N}$.
\end{proof}

The last remaining class of simplices are the internally additive simplices.  It will turn
out that certain kinds of internally additive simplices will cause trouble. 
Consider an internally additive simplex $\sigma$.  Write $\sigma = \{\l{v_0},\ldots,\l{v_p}\}$ where each $v_i$ is $F$-nonnegative and $\pm v_0 \pm v_1 \pm v_2 = 0$ for some choice of
signs.  If all three signs are the same, then we must have $F(v_0)=F(v_1)=F(v_2) = 0$, so
we can negate $v_0$ without changing the fact that $v_0$ is $F$-nonnegative.
The upshot is that we can assume that the three signs are not all the same.  Reordering
the $v_i$ if necessary, we can thus assume that $v_0 = v_1 + v_2$.
We call $\sigma$ a \emph{carrying simplex} if
%An internally additive simplex $\sigma$ is said to be a {\em carrying simplex} if it can be
%expressed as $\sigma = \{\l{v_0},\ldots,\l{v_p}\}$, where each $v_i$ is $F$-nonnegative and
%$v_0 = v_1 + v_2$ and
\begin{equation}
\label{eq:carrying}
\Big\lfloor \frac{F(v_0)}{N} \Big\rfloor \neq \Big\lfloor \frac{F(v_1)}{N} \Big\rfloor + \Big\lfloor \frac{F(v_2)}{N} \Big\rfloor.
\end{equation} 
To check that this is well-defined, observe that for the inequality \eqref{eq:carrying} to hold we must have $F(v_i)>0$ for $0 \leq i \leq 2$, in which case $v_0$
is uniquely determined since $F(v_0)$ is the maximum value among $\{F(v_0), F(v_1), F(v_2)\}$.
The following claim shows that the carrying simplices are the only possible source of trouble.

\begin{claims}
The map $\widehat{\pi}$ extends over the internally additive simplices $\sigma$ of $X$ that
are not carrying simplices.
\end{claims}
\begin{proof}[Proof of claim]
Write $\sigma = \{\l{v_0},\ldots,\l{v_p}\}$, where each $v_i$ is $F$-nonnegative and $v_0=v_1+v_2$. 
Since $\sigma$ is not a carrying simplex, we have
\[q_{v_0} = \Big\lfloor \frac{F(v_0)}{N} \Big\rfloor = \Big\lfloor \frac{F(v_1)}{N} \Big\rfloor + \Big\lfloor \frac{F(v_2)}{N} \Big\rfloor = q_{v_1} + q_{v_2}.\]
This implies that 
\[\pihat(\l{v_0})=\l{(v_0 - q_{v_0} w)}=\l{((v_1+v_2) - (q_{v_1}+q_{v_2}) w)}
=\l{((v_1-q_{v_1} w) + (v_2-q_{v_2} w) )}.\]
Therefore 
\begin{align*}
\{\widehat{\pi}(\l{v_0}),&\pihat(\l{v_1}),\pihat(\l{v_2}),\ldots,\widehat{\pi}(\l{v_p})\}\\
&= \{\l{((v_1-q_{v_1} w) + (v_2-q_{v_2} w) )}, \l{(v_1-q_{v_1} w)},\l{(v_2-q_{v_2} w)}\ldots,\l{(v_p - q_{v_p} w)}\}
\end{align*}
is an internally additive simplex of $X^{<N}$.
\end{proof}

It remains to deal with the carrying simplices. The key to our approach is the observation that even though the inequality \eqref{eq:carrying} may hold, the two sides never differ by more than $1$. Formally, this is the observation that the function
$\omega_N$ on $\Z^2$ defined by \[\Big\lfloor \frac{n_1+n_2}{N} \Big\rfloor = \Big\lfloor \frac{n_1}{N} \Big\rfloor + \Big\lfloor \frac{n_2}{N} \Big\rfloor + \omega_N(n_1,n_2),\] takes values \emph{only in $\{0,1\}$}. We remark that $\omega_N$ descends to a well-defined function\linebreak $\overline{\omega}_N\colon (\Z/N)^2 \rightarrow \{0,1\}$.
%via the surjection $\Z \rightarrow \Z/N$.
Regarding its image as lying in $\{0,1\}\subset \Z/N$, the function $\overline{\omega}_N$ is a group cocycle on
$\Z/N$ whose cohomology class is the Euler class of the nonsplit central
 extension
\[0 \longrightarrow \Z/N \longrightarrow \Z/N^2 \longrightarrow \Z/N \longrightarrow 0.\]
It is known as the \emph{carrying cocycle} because it records when carrying is necessary when adding modulo $N$; see Isaksen~\cite{IsaksenCarrying}.\medskip

Let $\Tri$ be the set of $2$-dimensional carrying simplices.  Define the simplicial complex $Y$ to be the result
of subdividing $X$ by adding a vertex $\tau_c$ to the center of each simplex $c \in \Tri$.
No carrying simplex can be contained in $X^{<N}$ since $\omega_N(0,-)=\omega_N(-,0)=0$. Thus the subcomplex $X^{<N}$ is not affected by this subdivision, so we can regard $X^{<N}$ as a subcomplex of $Y$.
Extend $\widehat{\pi}\colon X^{(0)} \rightarrow (X^{<N})^{(0)}$ to  $\widehat{\pi}\colon Y^{(0)} \rightarrow (X^{<N})^{(0)}$ as follows.
\begin{itemize}
\item Given $c \in \Tri$, write $c = \{\l{v_0},\l{v_1},\l{v_2}\}$, where each $v_i$ has $F(v_i)>0$ and $v_0 = v_1+v_2$.
The ordering of the $v_i$ is not canonical (as we mentioned above, $v_0$ is uniquely determined, but there is no way to distinguish $v_1$ and $v_2$), so simply make an arbitrary choice
for each $c$.
Define $\widehat{\pi}(\tau_c) = \l{(v_1 - q_{v_1}w - w)}$.
\end{itemize}
We first verify that $\pihat(\tau_c)$ lies in $X^{<N}$. By definition, $0\leq F(v_1-q_{v_1}w)<N$. However, the fact that $\omega_N(N\cdot n_1,-)=0$ means that a carrying simplex cannot have $F(v_1-q_{v_1}w)=0$ since this would imply $F(v_1)=q_{v_1}F(w)=q_{v_1}N$. Therefore $0< F(v_1-q_{v_1}w)<N$ and hence $-N< F(v_1-q_{v_1}w-w)<0$, as desired.

\begin{claims}
\label{claim:carrying}
The map $\widehat{\pi}$ extends over the images in $Y$ of carrying simplices $\sigma$ of $X$.
\end{claims}
\begin{proof}[Proof of claim]
Write $\sigma = \{\l{v_0},\ldots,\l{v_p}\}$, where each $v_i$ is $F$-nonnegative and $v_0 = v_1 + v_2$ 
and where for $c = \{\l{v_0},\l{v_1},\l{v_2}\}$ we have
$\widehat{\pi}(\tau_c) = \l{(v_1 + q_{v_1}w - w)}$.
By the definition of a carrying simplex, we have
\[q_{v_0} =  \Big\lfloor \frac{F(v_1)+F(v_2)}{N} \Big\rfloor =\Big\lfloor \frac{F(v_1)}{N} \Big\rfloor + \Big\lfloor \frac{F(v_2)}{N} \Big\rfloor + \omega_N(F(v_1),F(v_2)) = q_{v_1} + q_{v_2} + 1.\]  To simplify our notation,
for $0 \leq i \leq p$ we define $v'_i = v_i - q_{v_i}w$, so $\widehat{\pi}(\l{v_i}) = \l{(v'_i)}$.  Observe that $\widehat{\pi}(\tau_c) = \l{(v'_1 - w)}$ and
$v'_0 = v'_1 + v'_2 - w$.

The image of $\sigma$ in $Y$ consists of the three simplices
\[\alpha=\{\tau_c, \l{v_1}, \l{v_2},\l{v_3},\ldots,\l{v_p}\}, \quad\beta=
\{\l{v_0}, \tau_c, \l{v_2},\l{v_3},\ldots,\l{v_p}\}, \quad\gamma=
\{\l{v_0}, \l{v_1}, \tau_c, \l{v_3},\ldots,\l{v_p}\}.\]
We verify that $\widehat{\pi}$ extends over each of these in turn.   
For the first simplex $\alpha$ and the third simplex $\gamma$, we use $\pihat(\tau_c)=\l{(v'_1 - w)}$ to write
\begin{align*}
\pi(\alpha)&=\{\widehat{\pi}(\tau_c), \widehat{\pi}(\l{v_1}), \widehat{\pi}(\l{v_2}),\widehat{\pi}(\l{v_3}),\ldots,\widehat{\pi}(\l{v_p})\} \\
&= \{ \pihat(\tau_c), \l{(v'_1)}, \l{(v'_2)},\ldots,\l{(v'_p)}\}\\
&= \{\l{(v'_1-w)}, \l{(v'_1)}, \l{(v'_2)},\ldots,\l{(v'_p)}\}.\\
\pi(\gamma)&=\{\widehat{\pi}(\l{v_0}), \widehat{\pi}(\l{v_1}), \widehat{\pi}(\tau_c), \widehat{\pi}(\l{v_3}),\ldots,\widehat{\pi}(\l{v_p})\}\\
&= \{\l{(v'_0)}, \l{(v'_1)}, \pihat(\tau_c), \l{(v_3')},\ldots,\l{(v'_p)}\}\\
&= \{ \l{(v'_0)}, \l{(v'_1)}, \l{(v'_1-w)}, \l{(v_3')},\ldots,\l{(v'_p)}\}.
\end{align*}
These are both $\l{w}$-additive simplices of $X^{<N}$.
For the second simplex $\beta$, from $v'_0 = v'_1 + v'_2 - w$ we deduce the alternate identity $\widehat{\pi}(\tau_c) =\l{(v'_0-v'_2)}$, which we use to write
\begin{align*}
\pi(\beta)&=\{\widehat{\pi}(\l{v_0}), \widehat{\pi}(\tau_c), \widehat{\pi}(\l{v_2}),\widehat{\pi}(\l{v_3}),\ldots,\widehat{\pi}(\l{v_p})\}\\
&= \{\l{(v'_0)}, \pihat(\tau_c), \l{(v'_2)},\ldots,\l{(v'_p)}\}\\
&= \{\l{(v'_0)}, \l{(v'_0-v'_2)}, \l{(v'_2)},\ldots,\l{(v'_p)}\}.
\end{align*}
This is an internally additive simplex of $X^{<N}$. 
\end{proof}

Claims~\ref{claim:wstandard}--\ref{claim:carrying} demonstrate that $\widehat{\pi}\colon Y^{(0)} \rightarrow (X^{<N})^{(0)}$ extends over every simplex of $Y$, so it defines a simplicial retraction $\pi\colon Y\onto X^{<N}$. Since $Y$ is a subdivision of $X$, their realizations are homeomorphic, so this defines a topological retraction $\pi\colon X\onto X^{<N}$. This completes the proof of Proposition~\ref{prop:baretract2}.
\end{proof}

\subsection{The proof of Theorem~\texorpdfstring{\ref{theorem:basesacon}}{C'}}
\label{section:basesacon2proof}

We finally prove Theorem~\ref{theorem:basesacon}, which asserts that $\BA_n^m$ is CM of dimension $n$ for
$n \geq 1$ and $m \geq 0$ with $m+n \geq 2$.  The proof will be by induction on $n$.  

\para{Base case} We begin with the base case $n=1$. Our goal is to prove for $m\geq 1$ that $\BA_1^m$ is CM of dimension 1, i.e.\ is a connected nonempty graph. The vertices of the 1-dimensional complex $\BA_1^m$ are the vertices of $\B_1^m$, namely the lines spanned by vectors $w\in \Z^{m+1}$ 
such that $\{e_1,\ldots,e_m,w\}$ is a basis for $\Z^{m+1}$.  We can write such a vector as
\[w=a_1 e_1 + \cdots + a_m e_m \pm e_{m+1}\]
for some $a_i \in \Z$ and some sign.  Replacing $w$ with $-w$ changes the final sign, so we deduce that
the vertices of $\B_1^m$ are in bijection with elements $\abold \in \Z^m$ via the bijection that
takes $\abold = (a_1,\ldots,a_m)$ to the line $\l{v_{\abold}}$ with
\[v_{\abold}  = a_1 e_1 + \cdots a_m e_m + e_{m+1}.\]
Every 1-simplex of $\BA_1^m$ is externally additive since an internally additive simplex has dimension at least 2. 
Two lines $\l{v_{\abold}}$ and $\l{v_{\abold'}}$ determine an externally additive 1-simplex  
precisely when $\sign v_{\abold'} + \sign' v_{\abold} + \sign'' e_i = 0$ for some $1 \leq i \leq m$ 
and some $\sign, \sign', \sign'' = \pm 1$.  Examining the coefficient of $e_{m+1}$ in this
expression, we see that $\sign = \sign'$.  This implies that 
$\l{v_{\abold}}$ and $\l{v_{\abold'}}$ determine an externally additive 1-simplex exactly when
$\abold \in \Z^m$ and $\abold' \in \Z^m$ differ by a standard basis vector.

 We conclude that
$\BA_1^m$ is isomorphic to the Cayley graph of $\Z^m$ with respect to the generating set
$\{e_1,\ldots,e_m\}$, and is thus connected.  We remark that this is one point in the argument where working with lines and frames is essential; if we worked instead with primitive \emph{vectors} and \emph{bases}, we would obtain a disconnected graph consisting of two copies of this Cayley graph, one consisting of all vectors with $e_{m+1}$ coordinate 1 and the other consisting of those with coordinate $-1$.  This concludes the proof of the base case.

\para{Inductive step}
We now assume that $n>1$ and that
$\BA_{n'}^{m'}$ is CM of dimension $n'$ for all $1\leq n' <n$ and $m' \geq 0$ with $m'+n' \geq 2$. Under these assumptions, 
Proposition~\ref{prop:makecm} states that all links in $\BA_n^m$ are CM of the appropriate dimension, so it is enough to 
prove that $\BA_n^m$ is $(n-1)$-connected.

Fix $0 \leq p \leq n-1$, let $S^p$ be a combinatorial triangulation of a $p$-sphere, and let
$\phi\colon S^p \rightarrow \BA_{n}^{m}$ be a simplicial map.  Our goal is to show that $\phi$
can be homotoped to a constant map.
Let $F\colon \Z^{m+n} \onto \Z$ be the linear map taking $v \in \Z^{m+n}$ to the $e_{m+n}$-coordinate of $v$.
For a vertex $\l{v}$ of $\BA_n^m$, define $\rank(\l{v}) = \abs{F(v)}$; this is well-defined since
$\abs{F(v)} = \abs{F(-v)}$.  We then define
\[R(\phi) = \Max \Set{$\rank(\phi(x))$}{$x$ a vertex of $S^p$}.\]
This will be our measure of complexity for $\phi$.

If $R(\phi)=0$, then every simplex $\sigma$ of $\phi(S^p)$ is contained in the summand $\ker F$ of $\Z^{m+n}$. In particular, $\{\l{e_1},\ldots,\l{e_m}\}\ast\sigma$ is a partial augmented frame contained in $\ker F$; by Remark~\ref{remark:summandofsummand}, it is in fact a partial augmented frame for $\ker F$, so it can be extended to a partial augmented frame
for $\Z^{n+m}$ by adding the line $\l{e_{m+n}}$.  In other words, the entire image $\phi(S^p)$ is contained in the star of $\l{e_{m+n}}$ within $\BA_n^m$, so we can directly contract $\phi$ to the constant map whose image is the vertex $\l{e_{m+n}}$. 

We can therefore assume that $R(\phi)=R>0$.  The proof now is divided into four steps.  The end product of these four steps is that we can homotope
$\phi$ so as to decrease $R(\phi)$.  Repeating these steps over and over, we can eventually homotope $\phi$ so that
$R(\phi)=0$, at which point  we can contract $\phi$ directly to a constant map as above.

%\BeginSteps
\begin{step}
\label{step:1}
Given $\phi\colon S^p\to \BA_n^m$ with $R(\phi)\leq R$, we can homotope $\phi$ so that it satisfies the following Conditions~\ref{cond:11} and \ref{cond:21}.
\begin{compactenum}[label={\normalfont C\arabic*.},ref={C\arabic*}]
\item \label{cond:11} We still have $R(\phi) \leq R$.
\item \label{cond:21} If $\sigma$ is a simplex of $S^p$ such that $\phi(\sigma)$ is an additive simplex, then for all vertices $x$ of $\sigma$ we have $\rank(\phi(x)) < R$.
\end{compactenum}
\end{step}

Consider the following condition on a simplex $\sigma$ of $S^p$:
\begin{equation}
\label{eq:BAcond1}
\begin{aligned}
&\text{$\phi(\sigma)$ is an additive simplex, and}\\
&\text{some vertex $\l{v}$ of $\phi(\sigma)$ has $\rank(\l{v})=R$, and}\\
&\text{every vertex $\l{w}$ of $\phi(\sigma)$ either has $\rank(\l{w}) = R$ or lies in the additive core of $\phi(\sigma)$.}
\end{aligned}
\end{equation}
If $\phi$ does not satisfy Conditions~\ref{cond:11} and \ref{cond:21}, then there must
be some simplex $\sigma$ of $S^p$ satisfying \eqref{eq:BAcond1}.  
We can therefore
choose a simplex $\sigma$ of $S^p$ satisfying \eqref{eq:BAcond1} whose 
dimension $k$ is maximal among those satisfying
\eqref{eq:BAcond1}.  This maximality implies that $\phi$ takes $\Link_{S^p}(\sigma)$ to
$\Link_{\BA_n^m}(\phi(\sigma))^{<R}$.  

Let 
$\ell$ be the dimension of $\phi(\sigma)$; we certainly have $\ell \leq k$, but we 
might have $\ell < k$ if the restriction of $\phi$ to
$\sigma$ is not injective.  Lemma~\ref{lemma:connectlinksba}\ref{part:a} states that $\Link_{\BA_n^m}(\phi(\sigma))$ is isomorphic to $\B_{n-\ell}^{m+\ell}$, which
by Theorem~\ref{theorem:basescon} is CM of dimension $(n-\ell-1)$, and 
in particular is $(n-\ell-2)$-connected.  Using Lemma~\ref{lemma:baretract1},
we deduce that its retract $\Link_{\BA_n^m}(\phi(\sigma))^{<R}$ is $(n-\ell-2)$-connected.

By the definition of a combinatorial triangulation, 
$\Link_{S^p}(\sigma)$ is a combinatorial $(p-k-1)$-sphere.  
Since $p \leq n-1$ and $\ell \leq k$, we have $p-k-1 \leq n-\ell-2$, so $\phi|_{\Link_{S^p}(\sigma)}$ is null-homotopic within $\Link_{\BA_n^m}(\phi(\sigma))^{<R}$.
Using Zeeman's relative simplicial approximation theorem \cite{Zeeman}, we deduce 
that there exists a combinatorial $(p-k)$-ball $B$ with $\partial B \iso \Link_{S^p}(\sigma)$
and a simplicial map $\psi\colon B \rightarrow \Link_{\BA_n^m}(\phi(\sigma))^{<R}$ such that
$\psi|_{\partial B} = \phi|_{\Link_{S^p}(\sigma)}$.

The map $\psi$ extends to the $(p+1)$-ball $\sigma\ast B$ as $(\phi|_\sigma) \ast \psi \colon \sigma\ast B\to \B_n^m$.
The boundary of $\sigma\ast B$ is the union of the $p$-ball $\sigma\ast(\partial B)=\Star_{S^p}(\sigma)$, on which
$\phi|_\sigma\ast \psi=\phi|_{\Star_{S^p}(\sigma)}$, and the $p$-ball $(\partial \sigma)\ast B$.  We can thus homotope
$\phi$ across this $(p+1)$-ball to replace $\phi|_{\Star_{S^p}(\sigma)}$ with
$\phi|_{\partial \sigma}\ast\psi\colon (\partial \sigma)\ast B\to \B_n^m$.

The key property of this modification is that it eliminates the simplex 
$\sigma$ and does not add any other simplices satisfying \eqref{eq:BAcond1} or any
vertices mapping to vertices with $r(\l{v})\geq R$. Indeed, every new vertex lies in $B$, which maps to $\Link_{\BA_n^m}(\phi(\sigma))^{<R}$ by construction; this verifies the second claim. Moreover, every new simplex $\tau$ is the join of a simplex in $\partial\sigma$ with a nonempty simplex $\rho$ in $B$. Its image $\phi(\tau)$ is contained in $\phi(\sigma\ast\rho)=\phi(\sigma)\ast\phi(\rho)$. Thus $\phi(\tau)$ is only additive if it contains the additive core of $\phi(\sigma)$, in which case this is also the additive core of $\phi(\tau)$. Since $\phi(\rho)$ is disjoint from $\phi(\sigma)$ and every vertex has $r(\l{v})<R$ by construction, the new simplex $\tau$ cannot satisfy  \eqref{eq:BAcond1}.

Repeating this modification, we can homotope $\phi$ so that no simplex of $S^p$ satisfies \eqref{eq:BAcond1}, so $\phi$ satisfies Conditions~\ref{cond:11} and \ref{cond:21}, as desired.

\begin{step}
\label{step:2}
Given $\phi\colon S^p\to \BA_n^m$ satisfying Conditions~\ref{cond:11} and \ref{cond:21}, we can homotope $\phi$ so that it still satisfies the same Conditions~\ref{cond:12} and \ref{cond:22}, and additionally satisfies the following Condition~\ref{cond:3prime2}.
\begin{compactenum}[label={\normalfont C\arabic*.},ref={C\arabic*}]
\item \label{cond:12} We still have $R(\phi) \leq R$.
\item \label{cond:22} If $\sigma$ is a simplex of $S^p$ such that $\phi(\sigma)$ is an additive simplex, then for all vertices $x$ of $\sigma$ we have $\rank(\phi(x)) < R$.
\end{compactenum}
\begin{compactenum}[label={\normalfont C3\ensuremath{'}.}, ref={C3\ensuremath{'}}]
\item \label{cond:3prime2} If $x_1$ and $x_2$ are distinct vertices of $S^p$ such that $\rank(\phi(x_1)) = \rank(\phi(x_2)) = R$ and
$\phi(x_1) = \phi(x_2)$, then $x_1$ and $x_2$ are not joined by an edge in $S^p$.
\end{compactenum}
\end{step}

Although Conditions~\ref{cond:12} and \ref{cond:22} are restated multiple times in this section for convenience, we emphasize that these conditions are unchanged throughout.

Consider the following condition on a simplex $\sigma$ of $S^p$:
\begin{equation}
\label{eq:BAcond2}
\begin{aligned}
&\text{$\phi|_{\sigma}$ is not injective, and}\\
&\text{every vertex $\l{v}$ of $\phi(\sigma)$ has $\rank(\l{v})=R$}.
\end{aligned}
\end{equation}
If $\phi$ satisfies Conditions~\ref{cond:12} and \ref{cond:22} but not Condition~\ref{cond:3prime2}, then there must
be some simplex $\sigma$ of $S^p$ satisfying \eqref{eq:BAcond2}.
We can therefore choose a simplex $\sigma$ of $S^p$ satisfying \eqref{eq:BAcond2} whose
dimension $k$ is maximal among those satisfying
\eqref{eq:BAcond2}.  This maximality implies that $\phi$ takes $\Link_{S^p}(\sigma)$ to
$\Link_{\BA_n^m}(\phi(\sigma))^{<R}$.

In fact, even more is true.  Namely, Condition~\ref{cond:22} implies that if $\tau$ is a simplex of $\Link_{S^p}(\sigma)$,
then the simplex $\phi(\tau) \ast \phi(\sigma)$ of $\BA_n^m$ must be a standard simplex.  This
implies that $\phi$ actually takes $\Link_{S^p}(\sigma)$ to the subcomplex
$\Link_{\B_n^m}(\phi(\sigma))^{<R}$ of $\Link_{\BA_n^m}(\phi(\sigma))^{<R}$.  

Let
$\ell$ be the dimension of $\phi(\sigma)$; since $\phi|_{\sigma}$ is not injective,
we have $\ell \leq k-1$.  Theorem~\ref{theorem:basescon} says that $\B_n^m$ is CM
of dimension $(n-1)$, and hence $\Link_{\B_n^m}(\phi(\sigma))$ is
$(n-\ell-3)$-connected.  By Lemma~\ref{lemma:bretract}, its retract $\Link_{\B_n^m}(\phi(\sigma))^{<R}$ is $(n-\ell-3)$-connected.

The complex $\Link_{S^p}(\sigma)$ is a combinatorial $(p-k-1)$-sphere.   Since $p \leq n-1$ and $\ell \leq k-1$, we have $p-k-1 \leq n-\ell-3$, so $\phi|_{\Link_{S^p}(\sigma)}$ is null-homotopic within the subcomplex $\Link_{\B_n^m}(\phi(\sigma))^{<R}$ of $\BA_n^m$.
Therefore there exists a combinatorial $(p-k)$-ball $B$ with $\partial B \iso \Link_{S^p}(\sigma)$
and a simplicial map $\psi\colon B \rightarrow \Link_{\B_n^m}(\phi(\sigma))^{<R}$ such that
$\psi|_{\partial B} = \phi|_{\Link_{S^p}(\sigma)}$.

As in the previous step,
we can use this ball to homotope $\phi$ so as to replace $\phi|_{\Star_{S^p}(\sigma)}$ with
$\phi|_{\partial \sigma}\ast\psi\colon (\partial \sigma)\ast B\to \BA_n^m$.
The key property of this modification is that it eliminates the simplex
$\sigma$ and does not add any other simplices satisfying \eqref{eq:BAcond2}, while preserving Conditions~\ref{cond:12} and \ref{cond:22}. Indeed, every new vertex has $r(\l{v})<R$, so Condition~\ref{cond:12} is preserved. Every new simplex contains a new vertex, so it cannot satisfy \eqref{eq:BAcond2}. Finally, none of the simplices involved are additive since
the modifications in this step take place within the subcomplex $\B_n^m$, so Condition~\ref{cond:22} is preserved.
Repeating this, we can ensure that no simplices satisfy \eqref{eq:BAcond2} while preserving Conditions~\ref{cond:12} and \ref{cond:22}, as
desired.

\begin{step}
\label{step:3}
Given $\phi\colon S^p\to \BA_n^m$ satisfying Conditions~\ref{cond:12}, \ref{cond:22}, and \ref{cond:3prime2}, we can homotope $\phi$ so
that it satisfies the following Condition~\ref{cond:33}, as well as Conditions~\ref{cond:13} and \ref{cond:23}.
\begin{compactenum}[label={\normalfont C\arabic*.},ref={C\arabic*}]
\item \label{cond:13} We still have $R(\phi) \leq R$.
\item \label{cond:23} If $\sigma$ is a simplex of $S^p$ such that $\phi(\sigma)$ is an additive simplex, then for all vertices $x$ of $\sigma$ we have $\rank(\phi(x)) < R$.
\item \label{cond:33} If $x_1$ and $x_2$ are distinct vertices of $S^p$ such that $\Rank(\phi(x_1)) = \Rank(\phi(x_2)) = R$,
then $x_1$ and $x_2$ are not joined by an edge in $S^p$.
\end{compactenum}
\end{step}

If $\phi$ does not satisfy Condition~\ref{cond:33}, then there exists an edge
$e=\{x_1,x_2\}$ of $S^p$ with $\rank(\phi(x_1)) = \rank(\phi(x_2)) = R$.  Choose such an edge $e$. We will
homotope $\phi$ so as to eliminate $e$ without disturbing Conditions~\ref{cond:13}, \ref{cond:23}, or \ref{cond:3prime2}.

Choose $v_1,v_2 \in \Z^{m+n}$ with $F(v_1)=F(v_2)=R$ such that $\phi(x_1)=\l{v_1}$ and $\phi(x_2)=\l{v_2}$.  
Set $v_0= v_1 - v_2$, so $F(v_0) = 0$.  Condition~\ref{cond:3prime2} guarantees that 
$\l{v_1} \neq \l{v_2}$, so $v_0 \neq 0$.  Thus $\{\l{v_0},\l{v_1},\l{v_2}\}$ is an internally additive 
simplex and $\l{v_0}$ lies in $\Link_{\BA_n^m}(\{\l{v_1},\l{v_2}\})$.

Moreover, we claim that $\phi(\Link_{S^p}(e))$ is contained in the star of $\l{v_0}$ inside the subcomplex
$\Link_{\BA_n^m}(\{\l{v_1},\l{v_2}\})$.
To see this, consider an arbitrary simplex $\tau=\{\l{w_1},\ldots,\l{w_k}\}$ in $\phi(\Link_{S^p}(e))$.
Condition~\ref{cond:3prime2}  implies that $\tau$ is disjoint from $\{\l{v_1},\l{v_2}\}$, so $\tau$ lies in 
$\Link_{\BA_n^m}(\{\l{v_1},\l{v_2}\})$. By Condition~\ref{cond:23}, all simplices in $\phi(\Star_{S^p}(e))$ are 
standard, so in fact $\tau$ is a simplex of $\Link_{\B_n^m}(\{\l{v_1},\l{v_2}\})$.  This means that 
$\tau$ is contained is some frame $\{\l{e_1},\ldots,\l{e_m},\l{v_1},\l{v_2},\l{w_1},\ldots,\l{w_{n-2}}\}$ 
for $\Z^{m+n}$.  Then $\{\l{e_1},\ldots,\l{e_m},\l{v_0},\l{v_1},\l{v_2},\l{w_1},\ldots,\l{w_{n-2}}\}$ 
is an augmented frame, so $\tau$ is contained in the star of $\l{v_0}$, as desired.

Let $B$ be the cone on the combinatorial $(p-2)$-sphere $\Link_{S^p}(e)$. Since 
$\phi(\Link_{S^p}(e))$ is contained in the star of $\l{v_0}$, we can extend $\phi|_{\Link_{S^p}(e)}$ to
$\psi\colon B \rightarrow \Link_{\PartialBases_{n}^{m}}(\l{v})$ by sending the cone point to $\l{v_0}$.
As before, this lets us homotope $\phi$ to replace $\phi|_{\Star_{S^p}(e)}$ with $\phi|_{\partial e}\ast\psi\colon (\partial e)\ast B\to \BA_n^m$.
This eliminates the edge $e$. The only new vertex is $\l{v_0}$, so Conditions~\ref{cond:13} and \ref{cond:3prime2} are preserved since $F(v_0)=0$. Moreover, every new simplex in $(\partial e)\ast B$ contains 
$\l{v_0}$ but at most one of the lines $\l{v_1}$ and $\l{v_2}$.  Accordingly the 
new simplices are all standard, so Condition~\ref{cond:23} is preserved by this modification. Repeating this process lets us eliminate all edges 
violating Condition~\ref{cond:33}, as desired.

\begin{step}
\label{step:4}
Given $\phi\colon S^p\to \BA_n^m$ satisfying Conditions~\ref{cond:13}, \ref{cond:23} and \ref{cond:33}, we can homotope $\phi$ such 
that $R(\phi) < R$.
\end{step}

We remark that Condition \ref{cond:23} is not used in this step, though it is essential during Steps \ref{step:2} and \ref{step:3}.

If $R(\phi)=R$, then we can choose a vertex $x$ of $S^p$ such that $\rank(\phi(x)) = R$.  Write $\phi(x) = \l{v}$ with
$F(v) = R$.
Conditions~\ref{cond:13} and \ref{cond:33}
imply that $\phi$ takes $\Link_{S^p}(x)$ to $\Link_{\BA_n^m}(\l{v})^{<R}$.
We would like to apply Lemma~\ref{lemma:connectlinksba}\ref{part:b} and Proposition~\ref{prop:baretract2} to conclude that this subcomplex is highly connected. However, these results apply not to $\Link_{\BA_n^m}(\l{v})$, but to its proper subcomplex $\PLink_{\BA_n^m}(\l{v})$ defined in Definition~\ref{def:PLink}. Nevertheless, the vertices of the former that are excluded from the latter subcomplex are the $2m$ vertices of the form $\l{(v+e_i)}$ and $\l{(v-e_i)}$ for $i=1,\ldots,m$. Since $F(v+e_i)=F(v-e_i)=F(v)=R$, these complexes \emph{do} coincide when we restrict to vertices with $r(\l{w})<R$; in other words,  $\PLink_{\BA_n^m}(\l{v})^{<R}= \Link_{\BA_n^m}(\l{v})^{<R}$.

Lemma~\ref{lemma:connectlinksba}\ref{part:b} states
that $\PLink_{\BA_n^m}(\l{v})$ is isomorphic to $\BA_{n-1}^{m+1}$. Since $n>1$, our induction hypothesis states that this is CM of dimension $(n-1)$, and
in particular is $(n-2)$-connected.   Proposition~\ref{prop:baretract2} states that $\PLink_{\BA_n^m}(\l{v})$ admits a topological retraction onto $\PLink_{\BA_n^m}(\l{v})^{<R}$, so $\PLink_{\BA_n^m}(\l{v})^{<R}$ is $(n-2)$-connected as well.

The complex $\Link_{S^p}(x)$ is a combinatorial $(p-1)$-sphere.  
Since $p \leq n-1$, we have $p-1 \leq n-2$, so the restriction $\phi|_{\Link_{S^p}(\{x\})}$ is null-homotopic within $\PLink_{\BA_n^m}(\{\phi(x))\}^{<R}$. 
Just as in previous steps, this allows
us to homotope $\phi$ so as to eliminate $x$ without introducing any new vertices mapping
to vertices with $r(\l{w})=R$. This guarantees that Conditions~\ref{cond:13} %, \ref{cond:2}, 
and \ref{cond:33} are preserved by this modification.
Repeating this process lets us eliminate all vertices mapping to vertices with $r(\l{v})=R$, at which point $R(\phi)<R$, as desired.  This completes Step 4.
\medskip

Repeating the modifications of Steps~\ref{step:1}--\ref{step:4}, we can homotope $\phi$ so as to reduce $R(\phi)$ to $0$, at which point $\phi$ can be contracted to the constant map at the vertex $\l{e_{m+n}}$ as discussed before Step~\ref{step:1}. We conclude that an arbitrary map $\phi\colon S^p\to \BA_n^m$ for $0\leq p\leq n-1$ is null-homotopic, demonstrating that $\BA_n^m$ is $(n-1)$-connected. This completes the proof of Theorem~\ref{theorem:basesacon}.

\begin{remark}
\label{remark:PLMorsedifficulty}
%This remark is intentionally somewhat sketchy / conversational, because it's really just a tangential remark.
The hardest part of Theorem~\ref{theorem:basesacon} was showing that $\BA_n^m$ is $(n-1)$-connected. We proved this connectivity by defining an $\N$-valued function $r$ on the vertices of $\BA_n^m$, for which the subcomplex where $r(x)=0$ is contractible inside $\BA_n^m$, and showed that any sphere in $\BA_n^m$ of the appropriate dimension could be homotoped to lie in this subcomplex. From this outline the argument seems similar to ``PL Morse theory'' arguments, a common technique when proving such connectivity results. However, the structure of our proof is rather nonstandard and departs greatly from the PL Morse theory framework. Still, given the obvious similarities it is natural to wonder if our proof can be phrased in this language. 

Briefly, PL Morse theory tells us that if we can find a function $F\colon X^{(0)}\to \N$ on the vertices of a simplicial complex $X$ such that 
\begin{compactitem}
\item there are no ``horizontal edges'', i.e.\ edges $\{x,y\}$ with $F(x)=F(y)$, and
\item the ``descending link'' of each vertex $x$ with $F(x)>0$, i.e.\ the full subcomplex of the link spanned by vertices $y$
with ${F(y)<F(x)}$, is $(m-1)$-connected,
\end{compactitem} then the inclusion into $X$ of the subcomplex where $F(x)=0$ is $m$-connected.

However, our function $r$ is definitely not a PL Morse function; its descending links are not highly connected and it has many horizontal edges. In fact, we do not believe that any such PL Morse function can be defined on the vertices of $\BA_n^m$.

It turns out that it is possible to capture the argument in \S\ref{section:basesacon2proof} via a PL Morse function $F$, but only after passing to the barycentric subdivision of $\BA_n^m$. Unfortunately, this function is  rather unwieldy (see below). Moreover, to verify that the descending links have the appropriate connectivity requires recapitulating every step in \S\ref{section:proofbasescon}, \S\ref{section:introducebanm}, \S\ref{section:basesacon2links}, and \S\ref{section:basesacon2retract}, so ultimately this perspective would provide no benefit.

For the interested reader: after defining the function $r$ on vertices of $\BA_n^m$ as in our proof, for a simplex $\sigma\in \BA_n^m$ one should define $F(\sigma)\in \N\times \N\times \Z\times -\N$  by 
%
%%% TC: I intentionally left this as is, even though it's somewhat ugly (after all, our POINT is that it would be really ugly
%
\begin{align*}
F(\sigma) = \Big(\ \ &
\max_{x\text{ vertex of }\sigma}r(x),\\
&\text{\# of vertices of $\sigma$ realizing the maximum }\max_{x\in \sigma}r(x),\\
&\begin{cases}
-2&\text{ if $\sigma$ is additive and $\max_{x\in \sigma}r(x)$ is realized by a unique vertex}\\
-1&\text{ if $\sigma$ is internally additive and two of the three vertices}\\&\text{\qquad in its additive core realize the maximum $\max_{x\in \sigma}r(x)$}\\
0&\text{ if $\sigma$ is standard}\\
1&\text{ otherwise,}
\end{cases}\\
& -\dim(\sigma)\ \ 
\Big).
\end{align*}
If $\N\times \N\times \Z\times -\N$ is given the lexicographic order, then $F\colon \Poset(\BA_n^m)\to \N\times \N\times \Z\times -\N$ is a PL Morse function with well-ordered image whose descending links are $(n-2)$-connected.
\end{remark}

\begin{footnotesize}
\noindent
\begin{tabular*}{\linewidth}[t]{@{}p{\widthof{Department of Mathematics}+1.5in}@{}p{\linewidth - \widthof{Department of Mathematics}-1.5in}@{}}
{\raggedright
Thomas Church\\
Department of Mathematics\\
Stanford University\\
450 Serra Mall\\
Stanford, CA 94305\\
\href{mailto:tfchurch@stanford.edu}{\nolinkurl{tfchurch@stanford.edu}}}
&
{\raggedright
Andrew Putman\\
Department of Mathematics\\
University of Notre Dame\\
255 Hurley\\
Notre Dame, IN 46556\\
\href{mailto:andyp@nd.edu}{\nolinkurl{andyp@nd.edu}}}
\end{tabular*}\hfill
\end{footnotesize}

\end{document}